\documentclass[12pt,a4paper,english]{amsart}
\usepackage{mathptmx}
\usepackage[latin9]{inputenc}
\usepackage{color}
\usepackage{babel}
\usepackage{verbatim}
\usepackage{mathrsfs}
\usepackage{mathtools}
\usepackage{amstext}
\usepackage{amsthm}
\usepackage{amssymb}
\usepackage[unicode=true,pdfusetitle,
 bookmarks=true,bookmarksnumbered=false,bookmarksopen=false,
 breaklinks=false,pdfborder={0 0 1},backref=false,colorlinks=true]
 {hyperref}
\hypersetup{
 pdfborderstyle=,citecolor=blue}

\makeatletter

\pdfpageheight\paperheight
\pdfpagewidth\paperwidth

\numberwithin{equation}{section}
\numberwithin{figure}{section}
  \theoremstyle{plain}
  \newtheorem*{thm*}{\protect\theoremname}
  \theoremstyle{plain}
  \newtheorem*{cor*}{\protect\corollaryname}
  \theoremstyle{remark}
  \newtheorem*{rem*}{\protect\remarkname}
\theoremstyle{plain}
\newtheorem{thm}{\protect\theoremname}[section]
  \theoremstyle{plain}
  \newtheorem{lem}[thm]{\protect\lemmaname}
  \theoremstyle{plain}
  \newtheorem{prop}[thm]{\protect\propositionname}
  \theoremstyle{definition}
  \newtheorem{defn}[thm]{\protect\definitionname}
  \theoremstyle{remark}
  \newtheorem{rem}[thm]{\protect\remarkname}
  \theoremstyle{plain}
  \newtheorem{cor}[thm]{\protect\corollaryname}
  \theoremstyle{definition}
  \newtheorem*{example*}{\protect\examplename}
  \theoremstyle{definition}
  \newtheorem{example}[thm]{\protect\examplename}
  \theoremstyle{remark}
  \newtheorem*{claim*}{\protect\claimname}

\RequirePackage{tikz-cd}
\RequirePackage{amssymb}
\usetikzlibrary{calc}
\usetikzlibrary{decorations.pathmorphing}

\usepackage{babel}

\hypersetup{citecolor=blue}
\mathtoolsset{showonlyrefs}

\newcommand{\fif}{if and only if}

\newcommand{\bdl}{band-limited}

\newcommand{\beqa}{\begin{eqnarray*}}
\newcommand{\eeqa}{\end{eqnarray*}}

\DeclareMathOperator*{\supp}{supp}

\newcommand{\field}[1]{\mathbb{#1}}
\newcommand{\bR}{\field{R}}        
\newcommand{\bN}{\field{N}}        
\newcommand{\bC}{\field{C}}        
 \def\cF{\mathcal{F}}              

 \def\cH{\mathcal{H}}

 \def\cC{\mathcal{C}}

 \def\cO{\mathcal{O}}

\def\rd{\bR^d}

\def\rdd{{\bR^{2d}}}

\def\lrd{L^2(\rd)}

\def\<{\left<}
\def\>{\right>}

\def\mv1{M_v^1}

\newcommand{\rkhs}{reproducing kernel Hilbert space}

\newcommand {\norm}[2][2]{\left\Vert #2\right\Vert _{#1}}%

\providecommand{\corollaryname}{Corollary}
\providecommand{\definitionname}{Definition}

\providecommand{\lemmaname}{Lemma}
\providecommand{\propositionname}{Proposition}
\providecommand{\remarkname}{Remark}
\providecommand{\theoremname}{Theorem}
\renewcommand{\textemdash}{---}
\usepackage{microtype}

\makeatother

  \providecommand{\claimname}{Claim}
  \providecommand{\corollaryname}{Corollary}
  \providecommand{\definitionname}{Definition}
  \providecommand{\examplename}{Example}
  \providecommand{\lemmaname}{Lemma}
  \providecommand{\propositionname}{Proposition}
  \providecommand{\remarkname}{Remark}
  \providecommand{\theoremname}{Theorem}
\providecommand{\theoremname}{Theorem}

\begin{document}

\title[Critical Densities in Spectral Subspaces]{Necessary Density Conditions for Sampling and Interpolation in Spectral
Subspaces of Elliptic Differential Operators}

\author{Karlheinz Gr\"ochenig and Andreas Klotz}

\date{\today}
\begin{abstract}
We prove necessary density conditions for sampling in spectral subspaces of a second order uniformly elliptic differential operator on $\rd$ with slowly oscillating symbol. For constant coefficient operators, these are precisely Landaus necessary density conditions for band-limited functions, but for more general elliptic differential operators it has been unknown whether such a critical density even exists. Our results prove the existence of a suitable critical sampling density and compute it in terms of the geometry defined by the elliptic operator. In dimension $d=1$, functions in a spectral subspace can be interpreted as functions with variable bandwidth, and we obtain a new critical density for variable bandwidth. The methods are a combination of the spectral theory and the regularity theory of elliptic partial differential operators, some elements of limit operators, certain compactifications of $\rd$, and the theory of reproducing kernel Hilbert spaces. 
\end{abstract}



\thanks{This work was supported by the project P31887-N32 of the Austrian
Science Fund (FWF)}

\address{Faculty of Mathematics\\
 University of Vienna \\
Oskar Morgenstern-Platz 1 \\
A-1090 Vienna, Austria}

\email{karlheinz.groechenig@univie.ac.at}

\email{andreas.klotz@univie.ac.at}

\maketitle
\global\long\def\PW{PW_{\Omega}}
\global\long\def\cC{\mathcal{C}}
\global\long\def\rd{\mathbb{R}^{d}}
\global\long\def\r{\mathbb{\mathbb{R}}}
\global\long\def\n{\mathbb{\mathbb{N}}}
\global\long\def\z{\mathbb{Z}}
\global\long\def\Lii{L^{2}(\rd)}
\global\long\def\wam{L^{2}\left(L^{\infty}(Q)\right)}
\global\long\def\blii{\mathcal{B}(\Lii)}
\global\long\def\auto#1#2{\tau_{#1}(#2)}
\global\long\def\trans#1#2{T_{-#1}#2T_{#1}}

\global\long\def\inprod#1{\langle#1\rangle}
\global\long\def\abs#1{\left|#1\right|}
\global\long\def\specproj#1{\chi_{[0,\Omega]}(#1)}
\global\long\def\sobolev#1{W_{2}^{#1}}
\global\long\def\smoothind{\tilde{\chi}}
\global\long\def\ind{\chi}
\global\long\def\dom{\mathcal{D}}
\global\long\def\cinf{C_{c}^{\infty}\left(\rd\right)}
\global\long\def\cbinf{C_{b}^{\infty}\left(\rd\right)}
\global\long\def\cbinfd{C_{b}^{\infty}\left(\rd,\mathbb{C}^{d\times d}\right)}
\global\long\def\cbinfx#1{C_{b}^{\infty}\left(\rd,#1\right)}
\global\long\def\cub{C_{b}^{u}\left(\rd\right)}
\global\long\def\ll{L^{2}}
\global\long\def\so{C_{h}\left(\rd\right)}
\global\long\def\soi{C_{h}^{\infty}\left(\rd\right)}
\global\long\def\soid{C_{h}^{\infty}\left(\rd,\mathbb{C}^{d\times d}\right)}
\global\long\def\sosd{C_{\sigma}^{\infty}\left(\rd,\mathbb{C}^{d\times d}\right)}
\global\long\def\sod{C_{h}\left(\rd,\mathbb{C}^{d\times d}\right)}
\global\long\def\sox#1{C_{h}\left(\rd,#1\right)}
\global\long\def\soix#1{C_{h}^{\infty}\left(\rd,#1\right)}
\global\long\def\higson{h\rd}
\global\long\def\supp{\mathrm{supp\,}}
\global\long\def\c{\mathbb{C}}
\global\long\def\kl{k_{x}^{\Lambda}}
\global\long\def\Re{\mathfrak{Re}}
 \global\long\def\sos{C_{\sigma}\left(\rd\right)}
\global\long\def\tr{\mathrm{tr}}

\section{Introduction}

The classical Paley-Wiener\ space is the subspace $PW_{\Omega}=\{f\in L^{2}(\r):\supp\hat{f}\subseteq[-\Omega,\Omega]\}$
of $L^{2}(\r)$. Using Fourier inversion, one sees that the point
evaluation $f\mapsto f(x)$ is bounded on $PW_{\Omega}$. The fundamental
questions about $PW_{\Omega}$ are originally motivated by problems
in signal processing and information theory: when is $f\in PW_{\Omega}$
completely and stably determined by its samples $\{f(s):s\in S\}$
on a set $S\subseteq\r$? On which sets $S\subseteq\r$ can every
sequence $(a_{s})_{s\in S}\in\ell^{2}(S)$ be interpolated by a function
$f$ in $PW_{\Omega}$, so that $f(s)=a_{s}$ for all $s\in S$? These
questions were answered by Beurling~\cite{beur89} and Landau~\cite{landau67}. 
\begin{thm*}[A]
\label{tm0} (i) Assume that $S$ is uniformly separated and that
\begin{equation}
A\|f\|_{2}^{2}\leq\sum_{s\in S}|f(s)|^{2}\leq B\|f\|_{2}^{2}\qquad\text{ for all }f\in PW_{\Omega}\,,\label{eq:c1}
\end{equation}
then 
\begin{equation}
D^{-}(S)=\liminf_{r\to\infty}\inf_{x\in\r}\frac{\#(S\cap[x-r,x+r])}{2r}\geq\frac{\Omega}{\pi}\,.\label{eq:c2}
\end{equation}
(ii) If for all $a\in\ell^{2}(S)$ there exists $f\in PW_{\Omega}$,
such that $f(s)=a_{s},s\in S$, then 
\begin{equation}
D^{+}(S)=\limsup_{r\to\infty}\sup_{x\in\r}\frac{\#(S\cap[x-r,x+r])}{2r}\leq\frac{\Omega}{\pi}\,.\label{eq:c3}
\end{equation}
\end{thm*}
In the established terminology, a set that satisfies a sampling inequality
of the form~\eqref{eq:c1} is called a sampling set for the underlying
space $PW_{\Omega}$, or a set of stable sampling. A set on which
arbitrary $\ell^{2}$-data can be interpolated is called a set of
interpolation. The expressions $D^{-}(S)$ and $D^{+}(S)$ are called
the lower and the upper Beurling density.

The number $\Omega/\pi$ in \eqref{eq:c2} and \eqref{eq:c3} is an
important invariant of the space $PW_{\Omega}$ and has an interpretation
in information theory. Since, roughly speaking, the densities $D^{\pm}(S)$
measure the average number of samples in $S$ per unit length, the
\emph{necessary density conditions} of Theorem~\nameref{tm0} say
that at least $\Omega/\pi$ samples per unit length are required to
recover a function in $PW_{\Omega}$ from $f|_{S}$, whereas at most
$\Omega/\pi$ values per unit length are permitted to solve the interpolation
problem in $PW_{\Omega}$. Thus the density $\Omega/\pi$ represents
a critical value below which (stable) sampling is impossible, and
above which interpolation is impossible. Indeed, these questions about
sampling and interpolation were at the origin of Shannon's information
theory~\cite{Shannon48}, and the uniform sampling theorem with $S=\alpha\z$
is still considered the basis of analog-digital conversion in modern
signal processing. The ratio $D^{\pm}(S)/\Omega$ is a measure for
the redundancy, thus for the performance quality, of the sampling
set $S$. The theory of Beurling, Kahane, and Landau provides a rigorous
mathematical formulation for the existence of a critical density for
arbitrary sets $S$ (in place of $\alpha\z$). Although we will not
touch this question here, we mention that conditions of Theorem~\nameref{tm0}
yield almost a characterization of sets of sampling and of interpolation:
\emph{In dimension $d=1$, if $S$ is uniformly separated and $D^{-}(S)>1$,
then $S$ is a sampling set for $PW_{\Omega}$, and if $D^{+}(S)<1$,
then $S$ is a set of interpolation for $PW_{\Omega}$.} See~\cite{Kah62,beur89,seip04}
for an exposition of the sampling theory in the classical Paley-Wiener
space.

The connection with partial differential operators comes from the
observation that $PW_{\Omega}$ is a spectral subspace of the differential
operator $H=-\frac{d^{2}}{dx^{2}}$. This differential operator is
diagonalized by the Fourier transform, $\mathcal{F}(-\frac{d^{2}f}{dx^{2}})(\xi)=\xi^{2}\hat{f}(\xi)$,
so that the spectral projection on the interval $[0,\Omega]$ is given
by $\chi_{[0,\Omega]}(H)f=\mathcal{F}^{-1}(\chi_{[0,\Omega]}(\xi^{2})\hat{f})$.
This implies that 
\begin{equation}
\PW=\chi_{[0,\Omega^{2}]}(H)L^{2}(\r)\,.\label{eq:pw}
\end{equation}
This observation is the starting point for many generalizations of
Paley-Wiener spaces and sampling theorems. 

In this work we study the question of necessary density conditions
for sampling and interpolation in the spectral subspaces of a self-adjoint
uniformly elliptic differential operator 
\begin{equation}
H_{a}=-\sum_{j,k=1}^{d}\partial_{j}a_{jk}(x)\partial_{k}\,\label{eq:c4}
\end{equation}
acting on $\Lii$ with a smooth positive definite (matrix) symbol
$a=(a_{jk}(x))_{j,k=1,\dots,d}$. The Paley-Wiener space associated
to $H_{a}$ is the spectral subspace 
\[
PW_{\Omega}\left(H_{a}\right)=\chi_{[0,\Omega]}(H_{a})\lrd\,,
\]
where, as usual, $\chi_{[0,\Omega]}(H_{a})$ is the orthogonal projection
corresponding to the spectrum $[0,\Omega]$.

If the symbol $a(x)=a$ is constant, then $H_{a}$ is similar to the
Laplace operator, and the corresponding spectral subspace can be described
with Fourier techniques. For this case necessary density conditions
for sampling and interpolation are already contained in Landau's results~\cite{landau67}.
Optimal sufficient conditions for sampling in $\rd$ in terms of a
covering density were obtained by Beurling~\cite{beurling66}. However,
if $H_{a}$ is a uniformly elliptic differential operator with variable
coefficients, then the standard techniques break down, and it was
an open question (i) whether a critical density exists for sampling
and interpolation in the spectral subspaces of $H_{a}$, and (ii)
how to compute this critical density.

We will answer this question for a class uniformly elliptic operators.
We say $a\in\cbinfd$ is slowly oscillating if $\lim_{|x|\to\infty}|\partial_{k}a(x)|=0$
for $k=1,\dots,d$.
\begin{thm*}[B]
\label{tm1} If $a$ is slowly oscillating, then there exists a critical
density for sampling and interpolation for $\PW\left(H_{a}\right)$. 
\end{thm*}
Adapting the measure to the geometry associated to the differential
operator $H_{a}$,  the  critical density can be determined
explicitly. This is our main result. 
\begin{thm*}[C]
\label{tm3} Assume $H_{a}=-\sum_{j,k=1}^{d}\partial_{j}a_{jk}\partial_{k}$
is a self-adjoint uniformly elliptic operator with slowly oscillating
symbol $a\in\cbinfd$. Let $d\nu(x)=(\det a(x))^{-1/2}dx$ be the
associated measure.

(i) If $S\subseteq\rd$ is a set of stable sampling for $\PW\left(H_{a}\right)$
then 
\begin{equation}
D_{\nu}^{-}\left(S\right)=\liminf_{r\to\infty}\inf_{x\in\rd}\frac{\#(S\cap B_{r}(x))}{\nu(B_{r}(x))}\geq\frac{\left|B_{1}\right|}{\left(2\pi\right)^{d}}\Omega^{d/2}.\label{eq:5a-1}
\end{equation}
(ii) If $S\subseteq\rd$ is a set of interpolation for $\PW\left(H_{a}\right)$,
then 
\begin{equation}
D_{\nu}^{+}\left(S\right)=\limsup_{r\to\infty}\sup_{x\in\rd}\frac{\#(S\cap B_{r}(x))}{\nu(B_{r}(x))}\leq\frac{\left|B_{1}\right|}{\left(2\pi\right)^{d}}\Omega^{d/2}\,.\label{eq:5b-1}
\end{equation}
\end{thm*}
{} Except for the modified definition of the density, the formulation  of the
theorem is identical to Landau's theorem \cite{landau67}. By contrast, 
the method of proof is  vastly different as Fourier methods are not
available for the proof of Theorem \nameref{tm3}.  In addition
we draw the new insight that the appropriate notion of density must
be linked 
to this geometry.

For the special case of a symbol that is asymptotically constant at
infinity we can use the standard Beurling densities and obtain the
following consequence. 
\begin{cor*}[D]
\label{tm2} Assume that $a\in\cbinfd$ is asymptotically constant,
i.e., $\lim_{x\to\infty}a(x)=b$. Let $\Sigma_{\Omega}^{b}=\{\xi\in\rd:b\xi\cdot\xi\leq\Omega\}$.

(i) If $S\subseteq\rd$ is a set of sampling for the Paley-Wiener
space $\PW(H_{a})$, then 
\begin{equation}
D^{-}(S)\geq\frac{|\Sigma_{\Omega}^{b}|}{(2\pi)^{d}}\,.\label{eq:c5}
\end{equation}

(ii) If $S\subseteq\rd$ is a set of interpolation for the Paley-Wiener
space $\PW(H_{a})$, then 
\begin{equation}
D^{+}(S)\leq\frac{|\Sigma_{\Omega}^{b}|}{(2\pi)^{d}}\,.\label{eq:c5a}
\end{equation}
\end{cor*}
We note that the same critical density holds for the Paley-Wiener
space of the constant coefficient differential operator $H_{b}$.
Since $H_{a}$ may be considered a perturbation of $H_{b}$ and since
the Beurling density $D^{\pm}(S)$ is an asymptotic quantity, it is
to be expected that the necessary density for $\PW(H_{a})$ coincides
with the necessary density for $\PW(H_{b})$.

Let us put these statements into context.

\medskip{}

\textbf{Sampling in spectral subspaces.} Several researchers have
created an extensive qualitative theory of sampling in spectral subspaces
of a general unbounded, positive, self-adjoint operator $H$ on a
Hilbert space $\mathcal{H}$. In this case the abstract Paley-Wiener
space is defined as $PW_{[0,\Omega]}(H)=\chi_{[0,\Omega]}(H)\cH$.
Usually $\mathcal{H}=L^{2}(X,\mu)$ and $PW_{[0,\Omega]}(H)$ is a
reproducing kernel Hilbert space. In this situation many authors have
proved the existence of sampling sets~\cite{MR2970038,FFP16,FM11,Pes00,pes01,PesZay09}.
In particular the set-up of \cite{MR2970038,pes99,PesZay09} covers
the case of $H$ being a self-adjoint uniformly elliptic differential
operator on $L^{2}(\rd)$. The construction of sampling sets in these
abstract Paley-Wiener spaces requires some smoothness properties of
functions in $\PW(H)$ and a Bernstein inequality (see \eqref{eq:c11}
below). The result then is that a ``sufficiently dense'' subset
in $X$ is a sampling set and a ``sufficiently sparse'' subset of
$X$ is a set of interpolation. What remained unknown is the existence
of a critical density against which one could compare the quality
of the construction. Theorems~\nameref{tm1}, \nameref{tm3}, and
\nameref{tm2} address this gap for uniformly elliptic differential
operators.

Once a critical sampling density is established, one may aim for sampling
sets near the critical density. The question of optimal sampling sets
in spectral subspaces is wide open, in fact, it has become meaningful
only after the critical density is known explicitly. This problem
is already difficult for multivariate band-limited functions $PW_{K}=\{f\in\Lii:\supp\hat{f}\subseteq K\}$
for compact spectrum $K\subseteq\rd$ and was solved only recently
in \cite{MM10,OU08}. A possible general approach is via the construction
of Fekete sets and weak limits, as was carried out in \cite{GHOR19}
for Fock spaces with a general weight.

\textbf{Insight for partial differential operators.} Although the
spectral subspaces of a partial differential operator are natural
objects, they seem to have received little attention. To the best
of our knowledge, nothing is known about the nature of the corresponding
reproducing kernel and the behavior of functions in the spectral subspaces
$\PW$. Our investigation reveals several properties of the reproducing
kernel, such as the behavior of its diagonal and some form of off-diagonal
decay. These are key properties for the proofs of Theorems~\nameref{tm1}
and \nameref{tm3}, and we hope that these also hold some interest
for 
partial differential operators.

\textbf{Variable bandwidth.} Our original motivation comes from a
new concept of variable bandwidth. In \cite{cpam} we argued that
the spectral subspaces of the Sturm-Liouville operator $-\frac{d}{dx}a\frac{d}{dx}$
on $L^{2}(\r)$ for some function $a>0$ can be taken as spaces with
variable bandwidth. We proved that $a(x)^{-1/2}$ is a measure for
the bandwidth near $x$ (the largest active frequency at position
$x$). The function $a$ thus parametrizes the local bandwidth. For
constant $a=\Omega^{-2}$, the spectral subspace is just the classical
Paley-Wiener space $\PW$. For the special case of an eventually constant
parametrizing function $a$, i.e., $a$ is constant outside an interval
$[-R,R]$, we computed the critical density for sampling in $\PW(-\frac{d}{dx}a\frac{d}{dx})$.
The proof required intricate details of the scattering theory of one-dimensional
Schrödinger operators. Theorem~\nameref{tm3}, formulated for dimension
$d=1$, yields a significant extension of the density theorem for
the sampling of functions of variable bandwidth.
\begin{cor*}[E]
\label{cor1} Assume that $a\in C_{b}^{\infty}(\r)$ is bounded,
$a>0$, and $\lim_{x\to\pm\infty}a'(x)=0$. Let $\PW(H_{a})$ be the
Paley-Wiener space associated to $H_{a}$.

If $S$ is a sampling set for $\PW(H_{a})$, then 
\begin{equation}
D_{\nu}^{-}(S)\geq\frac{\Omega^{1/2}}{\pi}\,.\label{eq:17}
\end{equation}
Similarly, if $S$ is a set of interpolation for $\PW(H_{a})$, then
\[
D_{\nu}^{+}(S)\geq\frac{\Omega^{1/2}}{\pi}\,.
\]
\end{cor*}
\textbf{Methods.} The proofs of Theorems~\nameref{tm1} and \nameref{tm3}
combine ideas and techniques from several areas of analysis.

\emph{Critical density in reproducing kernel Hilbert spaces.} Originally,
density theorems in the style of Landau \textemdash{} and there are
dozens in analysis \textemdash{} were proved from scratch. In our
approach we apply the results on sampling and interpolation in general
reproducing kernel Hilberts spaces from \cite{densrkhs16}. The main
insight 
was that it suffices to verify some geometric conditions on the measure
space, such as a doubling condition of the underlying measure, and
of the reproducing kernel, such as some form of off-diagonal decay.
Once these conditions are satisfied, one obtains the existence of
a critical density and can calculate it in terms of the averaged trace
of the reproducing kernel. Since the geometric conditions are trivially
satisfied for $\rd$, our main technical difficulty is to understand
the reproducing kernel of the spectral subspaces of a self-adjoint
uniformly elliptic differential operator.

\emph{Regularity theory and heat kernel estimates.} To study this
reproducing kernel, we use the fundamental results of the regularity
theory of elliptic differential operators. With these tools we investigate
the smoothness of the reproducing kernel and compare various Sobolev
norms on $\PW(H_{a})$. See Lemma~\ref{lem:Sobolev} and Proposition~\ref{prop:kernbound}.
For an important technical detail (Proposition~\ref{prop:kernbound})
we will need Gaussian estimates for the heat kernel, which we expect
to play a key role in extensions of our theory.

\emph{Limit operators and slowly varying symbols.} To connect asymptotic
properties of the symbol $a$ of a partial differential operator $H_{a}$
to the spectral theory of $H_{a}$, we use the notion of limit operators.
Although we do not use any elaborate results from this theory (see~\cite{Georgescu11,MR2075882,Spakula17}),
limit operators are central to our arguments.

\emph{Higson compactification of $\rd$.} An important structure
underlying the proof of Theorem~C is a compactification of $\rd$,
the so-called Higson compactification. This is the compactification
arising as the maximal ideal space of the $C^*$-algebra of slowly
oscillating functions on $\rd$. By  Gelfand theory every slowly
oscillating function  can be identified with a  continuous function on
the Higson compactification $h\rd$, see, e.g.,
\cite{MR2075882,Roe03,Shteinberg00}. On a technical level we will show
that for slowly oscillating symbols the mapping $x\to T_{-x}k_x$ of
centered reproducing kernels  can be  extended continuously to the
compactification $h\rd$ (Proposition~\ref{prop: higs-cont}). 

\vspace{3mm}

The underlying philosophy is summarized in the following diagram.
We write $T_{x}f(z)=f(x-z)$ for the translation operator and $k_{x}$
for the reproducing kernel of $\PW(H_{a})$. 
\[
{\{T_{-x}a:x\in\rd\}\,\,\mathrm{compact}}\,\,\Longrightarrow\,\,{\{T_{-x}H_{a}T_{x}:x\in\rd\}\,\,\mathrm{compact}}\,\,\Longrightarrow\,\,{\{T_{-x}k_{x}:x\in\rd\}\,\,\mathrm{compact}}
\]
Thus, if $T_{x_{n}}a\to b$ in a suitable topology, then $T_{-x_{n}}H_{a}T_{x_{n}}\to H_{b}$
and the sequence of centered reproducing kernels $T_{-x_{n}}k_{x_{n}}$
converges to the reproducing kernel of $\PW(H_{b})$. In the considered
examples the limit operator $H_{b}$ is simpler than the original
operator $H_{a}$, and this facilitates information about the reproducing
kernel of $\PW(H_{a})$.

\vspace{2mm}


The paper is organized as follows. Section~2 prepares the background
material on regularity theory, symbol classes for partial differential
operators, and \rkhs s. We prove the basic properties of the Paley-Wiener
space $\PW(H_{a})$. Section~3 gives the precise formulation of the
general density theorem for $\PW(H_{a})$. Its proof is given in Sections~4
and 5. In Section~6 we calculate the critical density for sampling
in $\PW(H_{a})$ for the class of slowly varying symbols (Theorems~C
and D). We conclude with an outlook and collect additional material
in the appendix.

\section{Preliminaries}

\subsection{Notation\label{subsec:Notation}}

For a function $f$ on $\rd$ and $x,z\in\rd$ we define the translation
operator $T_{x}f=f\left(z-x\right)$. The open Euclidean ball of radius
$r$ at $x$ is $B_{r}(x)$, and $B_{r}=B_{r}(0)$.

We use standard multiindex notation, thus the differential operator
$D^{\alpha}$ is $\frac{\partial^{|\alpha|}}{\partial x_{1}^{\alpha_{1}}\dots\partial x_{d}^{\alpha_{d}}}$.

We will denote the space of uniformly continuous and bounded functions
on $\rd$ with values in a Banach space $X$ by $C_{b}^{u}\left(\rd,X\right)$.
The indices $c$, $\infty$, and $0$ refer to the subspaces of compactly
supported, smooth, and vanishing-at-infinity functions in $C(\rd)$.
Thus $C_{b}^{\infty}\text{\ensuremath{\left(\rd,X\right)}}$ consists
of all smooth $X$-valued functions with bounded derivatives of all
orders. 
The space $C^{\infty}(\rd,X)$ has the Fréchet space topology induced
by the seminorms $\abs f_{R,\alpha}=\sup_{x\in B_{R}\left(0\right)}\norm[X]{D^{\alpha}f(x)}$.
If $X=\c$, we 
write $\cbinf$, etc.

The Fourier transform of $f\in L^{1}(\rd)$ is 
\[
\mathcal{F}f\left(\omega\right)=(2\pi)^{-d/2}\int_{\rd}f(x)e^{-ix\cdot\omega}dx\,,
\]
and $\cF$ extends to a unitary operator on $\Lii$ as usual. For
every $s\geq0$ the Sobolev space $\sobolev s$ is defined by 
\[
\sobolev s=\left\{ f\in\lrd\colon\norm[\sobolev s]f=\left[(2\pi)^{-d/2}\int_{\rd}\abs{\hat{f}(\omega)}^{2}\left(1+\abs{\omega}^{2}\right)^{s}d\omega\right]^{1/2}<\infty\right\} \,.
\]
If $s\in\n$, then $\norm[\sobolev s]f\asymp\sum_{\abs{\text{\ensuremath{\alpha}}}\leq s}\norm{D^{\alpha}f}$.
By the Sobolev embedding theorem, $\sobolev s\hookrightarrow C_{0}(\rd)$
for $s>d/2$, and $\sobolev s$ is a reproducing kernel Hilbert space
with reproducing kernel $T_{x}\kappa$, $x\in\rd$, where $\widehat{\kappa}(\omega)=\widehat{\kappa}_{s}(\omega)=(1+\abs{\omega}^{2})^{-s}$.
This means that $f(x)=\langle f,T_{x}\kappa\rangle_{\sobolev s}$
for $f\in\sobolev s$. See, e.g., \cite{Wendland05}.

\subsection{The generalized Paley-Wiener Space and its basic properties\label{subsec:The-generalized-Paley}}

Pesenson's idea \cite{Pes98,pes01,PesZay09} was to define an abstract
Paley-Wiener space as a spectral subspace associated to an arbitrary
positive, self-adjoint operator $H\geq0$ with domain $\mathcal{D}(\mathcal{H})$
on a Hilbert space $\mathcal{H}$ and a spectral interval $[0,\Omega]$.
Let $\specproj H$ be the spectral projection of $H$, then the generalized
Paley-Wiener space is defined as 
\begin{equation}
\PW\left(H\right)=\specproj H\mathcal{H}\,.\label{eq:PW}
\end{equation}
Equivalently, for a positive, self-adjoint operator, one can define
the Paley-Wiener space $\PW(H)$ by a Bernstein inequality~\cite{grkl10,pes01,PesZay09}:
$f\in\PW(H)$, \fif\ $f\in\mathcal{D}(H^{k})$ for all $k\in\bN$,
and 
\begin{equation}
\bigl\Vert H^{k}f\bigr\Vert_{2}\leq\Omega^{k}\norm{f}\qquad\text{ for all }k\in\bN\,.\label{eq:c11}
\end{equation}

If $H=-\frac{d^{2}}{dx^{2}}$ on $L^{2}(\bR)$ , then 
\[
\PW(H)=\{f\in L^{2}(\bR):\supp\hat{f}\subseteq[-\sqrt{\Omega},\sqrt{\Omega}]\}
\]
is precisely the classical Paley-Wiener space, or in engineering language
the space of \bdl\ functions with bandwidth $2\sqrt{\Omega}$.

\textbf{Convention}. In this work we consider positive, formally self-adjoint
differential expressions $H=H_{a}$ of the form 
\begin{equation}
H_{a}f=-\sum_{j,k=1}^{d}\partial_{j}a_{jk}\partial_{k}f,\quad f\in\sobolev 2.\label{eq:rev1}
\end{equation}
Here the \emph{matrix symbol} $a\in\cbinfd$ is positive definite,
i.e., $a_{jk}=\overline{a_{kj}}\in C_{b}^{\infty}\text{\ensuremath{\left(\rd\right)}}$
and $a(x)\xi\cdot\xi\geq\theta\abs{\xi}^{2}$ for $\xi,x\in\rd$ and
some fixed $\theta>0$. Then $H_{a}$ is a positive, uniformly elliptic
self-adjoint operator on $\rd$ with domain $\dom\left(H_{a}\right)=\sobolev 2$.
In particular $\cinf$ is a core for $H_{a}$, i.e., $H_{a}$ is the
operator closure of $H_{a}|_{\cinf}$. The regularity theory of elliptic
differential operators asserts that for every $k\in\n_{0}$ there
is a $c_{k}\in\r$ such that 
\begin{equation}
H_{a}^{k}+c_{k}\colon\sobolev{2k}\to\Lii\label{eq:ellreg}
\end{equation}
is a Hilbert space isomorphism. See~\cite[Thm 6.3.12]{Zimmer90}
or the standard references \cite{Agmon65,MR1205177}. 
For further use we record the fact that a uniformly elliptic operator
is one-to-one on its domain and thus 
\begin{equation}
0\text{ is not an eigenvalue of }H_{a}.\label{eq:notev}
\end{equation}
For this, recall that $H_{a}f=0$ implies that $f\in C^{\infty}(\rd)$.
Using the ellipticity and $f\in\sobolev 2$, the identity $\langle H_{a}f,f\rangle=\int\sum_{j,k}a_{jk}\partial_{k}f(x)\overline{\partial_{j}f(x)}\,dx=0$
implies that $\partial_{j}f\equiv0$, thus $f=0$. 

\begin{rem*}
We regard the mapping $a\mapsto H_{a}$ as a mapping from functions
to operators (a symbolic calculus) and refer to $a$ as the (matrix)
symbol of the operator. This terminology differs slightly from the
usage in PDE, where the (principal) symbol of the differential operator
$\sum_{\abs{\alpha}\leq m}a_{\alpha}D^{\alpha}$ is the function $p\left(x,\xi\right)=\sum_{\abs{\alpha}=m}a_{\alpha}\xi^{\alpha}$
on $\rdd$. For the second order differential operator $H_{a}$ in
\eqref{eq:rev1} the principal symbol is $p\left(x,\xi\right)=a(x)\xi\cdot\xi$.
Since $H_{a}$ is self-adjoint the coefficients $a_{\alpha}$ are
all real for $\abs{\alpha}=2$. 
\end{rem*}
First we verify that $\PW\left(H_{a}\right)$ embeds in every Sobolev
space. 
\begin{lem}
\label{lem:Sobolev} The Paley-Wiener space $\PW(H_{a})$ is continuously
embedded in all Sobolev spaces $\sobolev s,s\geq0,$ and in $C_{0}^{\infty}(\rd)$.
As a consequence, on $\PW\left(H_{a}\right)$, the $L^{2}$-norm and
the Sobolev norms are equivalent. 
\end{lem}

\begin{proof}
Let $f\in\PW\left(H_{a}\right)$ and $k\in\n$. By elliptic regularity
and Bernstein's inequality \eqref{eq:c11}, $\norm[{\sobolev{2k}}]{f}\asymp\norm{(H^{k}+c_{k})f}\leq(\Omega^{k}+\abs{c_{k}})\norm{f}\,.$
Consequently, $f\in\bigcap_{k\in\n}\sobolev{2k}=\bigcap_{s\geq0}\sobolev s\subseteq C_{0}^{\infty}(\rd)$
via the Sobolev embedding. 
\end{proof}
Next we show that $\PW\left(H_{a}\right)$ is a reproducing kernel
Hilbert space in $\lrd$. Recall that a reproducing kernel Hilbert
space $\mathcal{H}$ is a Hilbert space of functions defined on a
set $X$ such that $f\left(x\right)=\inprod{f,k_{x}}_{\mathcal{H}}$
for all $f\in\mathcal{H}$ and $x\in X$. We write $k(x,y)=\overline{k_{x}(y)}$
for the reproducing kernel of $\cH$. See, e.g. \cite{Aron50}. 
\begin{prop}
\label{prop:kernbound} 
 There exists a reproducing kernel $k_{x}\in\PW(H_{a})$, such that
$\specproj{H_{a}}f\left(x\right)=\inprod{f,k_{x}}$ for all $f\in\Lii$
and all $x\in\rd$. In addition, there are positive constants $c,C$,
such that 
\begin{equation}
0<c\leq\norm{k_{x}}\leq C\quad\text{for all }x\in\rd\,.\label{eq:kernelbound}
\end{equation}
\end{prop}

\begin{proof}
Let $f\in\PW\left(H_{a}\right)$ and $s>d/2$. By Lemma \ref{lem:Sobolev},
$f\in\sobolev s$ and $\norm{f}\asymp\norm[\sobolev s]{f}$. Since
$\sobolev s$ is a \rkhs, we obtain 
\begin{equation}
\abs{f\left(x\right)}=\abs{\inprod{f,T_{x}\kappa}_{\sobolev s}}\leq\norm[\sobolev s]{T_{x}\kappa}\,\norm[\sobolev s]f\leq C\,\|f\|_{2}\,.\label{eq:l2inlinf}
\end{equation}
Thus $\PW\left(H_{a}\right)$ is a reproducing kernel Hilbert space
with kernel $k_{x}\in\PW(H_{a})$.

For the lower bound in (\ref{eq:kernelbound}) we do not have a proof
based exclusively on regularity theory. Instead we refer to \cite[Lemma 3.19]{MR2970038}
where the lower bound for the reproducing kernel was derived by means
of heat kernel estimates. 
As some details and notation differ, we reproduce the proof in the
appendix. 
\end{proof}
\begin{prop}
\label{cor:normcont} The mapping $x\mapsto k_{x}$ is continuous
from $\rd$ to $\sobolev s$, $s\geq0$. 
\end{prop}

\begin{proof}
Since $k_{x}\in\PW(H_{a})\subset C_{0}^{\infty}$$\left(\rd\right)$,
we obtain 
\begin{align*}
\norm[\sobolev s]{k_{x}-k_{y}}^{2}\leq C\norm{k_{x}-k_{y}}^{2} & =C\left(k_{x}\left(x\right)-k_{x}\left(y\right)-k_{y}\left(x\right)+k_{y}\left(y\right)\right)\to0
\end{align*}
for $y\to x$, where the first inequality follows from 
Lemma \ref{lem:Sobolev}. 
\end{proof}

\subsection{Sampling and interpolation in $\protect\PW(H_{a})$ and the Beurling
Densities.\label{subsec:Sampling-and-interpolation}}

Let $\mu$ be a Borel measure on $\rd$ that is equivalent to Lebesgue
measure in the sense that $d\mu=h\,dx$ for a measurable function
$h$ with $0<c<h(x)<C$ for all $x\in\rd$.

The lower Beurling density of $S$ with respect to $\mu$ is defined
as 
\begin{equation}
D_{\mu}^{-}(S)=\liminf_{r\to\infty}\inf_{x\in\rd}\frac{\#(S\cap B_{r}(x))}{\mu(B_{r}(x))}\,,\label{eq:2low}
\end{equation}
and the upper Beurling density of $S$ is 
\begin{equation}
D_{\mu}^{+}(S)=\limsup_{r\to\infty}\sup_{x\in\rd}\frac{\#(S\cap B_{r}(x))}{\mu(B_{r}(x))}\,.\label{eq:2up}
\end{equation}
If $d\mu=dx$ we omit the subscript and write $D^{\pm}\left(S\right)$.

For sampling in reproducing kernel Hilbert spaces the relevant measure
is $d\mu\left(x\right)=k\left(x,x\right)dx$. We call the Beurling
density with respect to this measure the dimension-free density and
write $D_{0}^{\pm}(S)$ for $D_{\mu}^{\pm}(S)$ .

We say that the reproducing kernel $k$ of a \rkhs\ $\cH$ satisfies
the \emph{weak localization} property (WL), if \\
 \label{WL}\emph{(WL)}: for every $\varepsilon>0$ there is a constant
$r=r(\varepsilon)$, such that 
\begin{equation}
\sup_{x\in\rd}\int_{\rd\setminus B_{r}(x)}\abs{k(x,y)}^{2}d\mu(y)<\varepsilon^{2}.\tag{{WL}}\label{eq:hapcont}
\end{equation}
The discrete analog of the weak localization is the so-called \emph{homogeneous
approximation property} (HAP) of the reproducing kernel: 

\emph{(HAP)}:\label{HAP} Assume that $S$ is such that $\{k_{s}:s\in S\}$
is a \emph{Bessel sequence} for $\cH$, 
i.e., $S$ satisfies the upper sampling inequality 
$\sum_{s\in S}\abs{f\left(s\right)}^{2}\leq C\norm f^{2}$ for all
$f\in\cH$. 
Then for every $\varepsilon>0$ there is a constant $r=r(\varepsilon)$,
such that 
\begin{equation}
\sup_{x\in\rd}\sum_{s\in S\setminus B_{r}(x)}\abs{k(x,s)}^{2}<\varepsilon^{2}\,.\tag{{HAP}}\label{eq:hapclass}
\end{equation}

In general, an upper sampling inequality implies that 
for some (and hence all) $\rho>0$ 
\[
\max_{x\in\rd}\#(S\cap B_{\rho}(x))<\infty\,.
\]
We call such a set $S$ \emph{relatively separated}. See~\cite[Lemma~3.7]{densrkhs16}.

The two localization properties \eqref{eq:hapcont} and \eqref{eq:hapclass}
are the key properties of the reproducing kernel required for an abstract
density theorem to hold. For \rkhs s embedded in $\lrd$ this can
be stated as follows~\cite[Cor.~4.1]{densrkhs16} . 
\begin{thm}
\label{abstrdens} Let $\cH\subseteq L^{2}(\rd,\mu)$ be a \rkhs\ with
kernel $k$. Assume that $k$ satisfies the boundedness property \eqref{eq:kernelbound}
on the diagonal, the weak localization (\ref{eq:hapcont}) and the
homogeneous approximation property (\ref{eq:hapclass}). 

(i) If $S$ is a sampling set for $\cH$, then $D_{0}^{-}(S)\geq1$and

(ii) If $S$ is an interpolating set for $\cH$, then $D_{0}^{+}(S)\leq1$
.
\end{thm}

This result holds under a set of natural assumptions on metric measure
spaces and conditions on the reproducing kernel. We will not dwell
on the geometric conditions, e.g., doubling measure, as these are
clearly satisfied for $\rd$ with $\mu$ equivalent to 
Lebesgue measure. We want to verify Theorem \ref{abstrdens} for $\cH=\PW$$\left(H_{a}\right)$
for a suitable class of symbols $a$. %
The boundedness of the diagonal of the kernel was already established
in Proposition~\ref{prop:kernbound}, \eqref{eq:kernelbound}. To
prove Theorems B and C we therefore need to verify the properties
(\ref{eq:hapcont}) and (\ref{eq:hapclass}) for the \rkhs\ $\PW(H_{a})$.

Observe that (\ref{eq:hapcont}) is equivalent to the condition 
\begin{equation}
\sup_{x\in\rd}\int_{\rd\setminus B_{r}(0)}\abs{T_{-x}k_{x}(y)}^{2}dy<\varepsilon^{2}\label{eq:rk}
\end{equation}
for the \emph{centered} reproducing kernels. We will show the stronger
statement that the set $\left\{ T_{-x}k_{x}\colon x\in\rd\right\} $
is relatively compact in $\Lii$. The Riesz-Kolomogorov compactness
theorem then implies \eqref{eq:rk} and thus (\ref{eq:hapcont}).

The proof of (\ref{eq:hapclass}) requires some additional local regularity
of $k_{x}$. We will use prominently the elliptic regularity theory
to show that $\left\{ T_{-x}k_{x}\colon x\in\rd\right\} $ is relatively
compact in all Sobolev spaces $\sobolev s$. For proof of (\ref{eq:hapclass})
it is fundamental that the point evaluation on $\PW$$\left(H_{a}\right)$
can be expressed twofold as 
\begin{equation}
f(x)=\langle f,k_{x}\rangle_{L^{2}}=\langle f,T_{x}\kappa\rangle_{\sobolev s}\quad\text{for all }f\in\PW\left(H_{a}\right)\,.\label{eq:c12-1}
\end{equation}

\subsection{Classes of symbols, limit operators}

First we define the relevant symbol classes. Let 
\begin{equation}
\auto x{H_{a}}=T_{-x}H_{a}T_{x}=H_{T_{-x}a}\,,\label{eq:auto}
\end{equation}
be the conjugation of $H_{a}$ by the translation $T_{x}$. If $a\in\cbinfd$
observe that $\auto x{H_{a}}$ is again a self-adjoint, uniformly
elliptic operator with domain $\sobolev 2$ and core $\cinf$. In
this section we describe symbol classes that ensure that $\left\{ \tau_{x}(H_{a})f\colon x\in\rd\right\} $
is relatively compact in $\Lii$ for all $f\in\cinf$. Equivalently,
every sequence $\auto{x_{k}}{H_{a}}f$ has a norm convergent subsequence.
If $\left(x_{k}\right)$ is bounded, this follows from the continuity
of $x\mapsto T_{x}f$. To treat unbounded sequences we need some terminology. 
\begin{defn}
\label{def:limop}Assume $a\in\cbinfd$. If the net $\left(x_{\lambda}\right)_{\lambda\in\Lambda}\subset\rd$
diverges to infinity and there is an operator $H\in\mathcal{B}\left(\sobolev 2,\Lii\right)$
such that $\lim_{\lambda}\auto{x_{\lambda}}{H_{a}}f=Hf$ for all $f\in\cinf$,
then we call $H$ a \emph{limit operator} of $H_{a}$. 
\end{defn}

\begin{rem}
(i) Existence and uniqueness of the limit operator follow from the
theorem of Banach-Steinhaus.

(ii) We do not even scratch the surface of the method of limit operators:
see, amongst many others, \cite{MR2075882,Spakula17,Rabinovich04b},
in $C^{*}$-algebra setting \cite{DavGeorg13,Georgescu11,Georgescu18}.

(iii) Limit operators are related to compactifications of $\rd$.
An example can be found in Section \ref{subsec:The-Higson-compactification}. 
\end{rem}

\subsubsection{Compact orbits}

Identity~\eqref{eq:auto} suggests that compactness properties of
$\left\{ \auto x{H_{a}}\colon x\in\rd\right\} $ are related to compactness
properties of $\left\{ T_{-x}a\colon x\in\rd\right\} $, so we investigate
these first. In Section \ref{sec:Geometric-Beurling-Densities} we
will deal with a  non-metrizable  compactification of $\rd$, 
therefore we formulate most  results for nets
$\left(x_{\lambda}\right)_{\lambda\in\Lambda}$ instead of sequences.
\begin{lem}
\label{prop:Arzela} (i) If $f\in C_{b}^{\infty}\left(\rd\right)$
then 
 $\left\{ T_{x}f\colon x\in\rd\right\} $ is relatively compact in
the Fréchet space $C^{\infty}(\rd)$ with respect to its topology
of uniform convergence of all derivatives on compact sets.

(ii) In particular, if $\lim_{\lambda}T_{x_{\lambda}}f=g$ pointwise,
then $\lim_{\lambda}T_{x_{\lambda}}f=g$ in $C^{\infty}\left(\rd\right)$.
The limit function $g$ is in $\cbinf$ again. 
\end{lem}

\begin{proof}
(i) The space $C^{\infty}(\rd)$ has the Heine Borel property \cite[1.46]{Rudin73},
so it suffices to verify that $\left\{ T_{x}f\colon x\in\rd\right\} $
is bounded in $C^{\infty}(\rd)$, which means that 
\[
\norm[L^{\infty}(B_{r}(0))]{D^{\alpha}T_{x}f}<C_{\alpha,r}\,\text{ for all }x\in\rd\text{ and all }r>0,\alpha\in\n^{d}.
\]
But this is trivial for $f\in C_{b}^{\infty}(\rd)$, since all derivatives
are globally bounded.

(ii) We apply the following observation: A net converges to a limit
$g$ if and only if every subnet has a subnet that converges to $g$.
By (i) every subnet of $\left(T_{x_{\lambda}}f\right)_{k\in\n}$ has
a subnet $\left(T_{z_{\lambda}}f\right)_{k\in\n}$ that converges
in $C^{\infty}\left(\rd\right)$ (to the limit function $g$). We
conclude that $\left(T_{x_{\lambda}}f\right)_{k\in\n}$ converges
to $g$ in $C^{\infty}\left(\rd\right)$. As all functions and their
derivatives of all orders are bounded and continuous, this is true
for the limit as well. 
\end{proof}
\begin{prop}
\label{strongcomp}Let $a\in\cbinfd,\,k,m\in\n_{0}$, and assume $\lim_{\lambda}T_{-x_{\lambda}}a=b$
pointwise. Then, for every $f\in\sobolev{2m+2k}$ 
\begin{equation}
\lim_{\lambda}\norm[\sobolev{2m}]{\left(\auto{x_{\lambda}}{H_{a}^{k}}-H_{b}^{k}\right)f}=0\,.\label{eq:limopcont}
\end{equation}
\end{prop}

\begin{proof}
We treat the case $k=1$ first and assume for the moment that $f\in\cinf$.
Set $a^{\left(\lambda\right)}=T_{-x_{\lambda}}a$. We can express
$H_{a}$ in the form $H_{a}=\sum_{\abs{\beta}\leq2}a_{\beta}D^{\beta}$
with coefficients $a_{\beta}\in\cbinf$, and estimate, for every multindex
$\alpha$ with $|\alpha|\leq2m$, 
\[
\abs{D^{\alpha}\left(H_{a^{(\lambda)}}-H_{b}\right)f}=\abs{D^{\alpha}\sum_{\abs{\beta}\leq2}(a_{\beta}^{(\lambda)}-b_{\beta})D^{\beta}f}=\abs{\sum_{\abs{\beta}\leq2}\sum_{\abs{\gamma}\leq\abs{\alpha}}\binom{\alpha}{\gamma}D^{\gamma}\left(a_{\beta}^{(\lambda)}-b_{\beta}\right)D^{\alpha-\gamma+\beta}f}\,.
\]
By Lemma \ref{prop:Arzela} we have $\lim_{\lambda}D^{\alpha}a^{(\lambda)}=D^{\alpha}\text{\ensuremath{b}}$
uniformly on compact sets, so the convergence is actually uniform
on $\supp f$, and thus 
\begin{equation}
\lim_{\lambda}\norm[\infty]{D^{\alpha}\left(H_{a^{(\lambda)}}-H_{b}\right)f}=0\,.\label{eq:cinfc}
\end{equation}
Consequently 
\begin{align}
\norm[\sobolev{2m}]{\left(H_{a^{(\lambda)}}-H_{b}\right)f} & \leq C\max_{\abs{\alpha}\leq2m}\norm{D^{\alpha}\left(H_{a^{(\lambda)}}-H_{b}\right)f}\label{eq:cinfconv}\\
 & \leq C\abs{\supp f}^{1/2}\max_{\abs{\alpha}\leq2m}\norm[\infty]{D^{\alpha}\left(H_{a^{(\lambda)}}-H_{b}\right)f}\to0\,.\nonumber 
\end{align}
As $\cinf$ is dense in $\sobolev{2m+2}$, and the operators $H_{a^{\left(\lambda\right)}}$
are uniformly bounded from $\sobolev{2m+2}$ to $\sobolev{2m}$, a
standard density argument (see, e.g., \cite[Lemma 1.14]{Teschl09})
implies $\norm[\sobolev{2m}]{\left(H_{a^{(\lambda)}}-H_{b}\right)f}\to0$
for all $f\in\sobolev{2m+2}$.

For $k>1$ observe that 
\[
H_{a}^{k}f-H_{b}^{k}f=H_{a}^{k-1}\left(H_{a}f-H_{b}f\right)+\left(H_{a}^{k-1}f-H_{b}^{k-1}f\right)H_{b}f\,.
\]
As $\lim_{\lambda}\norm[\sobolev{2m}]{\left(H_{a^{(\lambda)}}-H_{b}\right)f}=0$
for $f\in\sobolev{2m+2}$, the result follows by induction on $k$. 
\end{proof}
\begin{rem*}
The statement of the proposition and its proof are valid under the
following more general conditions: $a_{\lambda},b\in\cbinf$, $a_{\lambda}\xrightarrow{C^{\infty}}b$,
and $\left(H_{a_{\lambda}}\right)$ is uniformly bounded from $\sobolev 2$
to $\Lii$. 
\end{rem*}
Though not needed in the sequel, we state an interesting corollary
that shows how compactness properties of the orbit $\{T_{x}a:x\in\rd\}$
are transferred to compactness properties of $\{\tau_{x}(H_{a}):x\in\rd\}$. 
\begin{cor}
\label{cor:rc} If $a\in\cbinfd$ and $f\in\cinf$ the set $\left\{ \auto x{H_{a}}f\colon x\in\rd\right\} $
is relatively compact in every Sobolev space $\sobolev s$, $s>0$. 
\end{cor}

\begin{proof}
The set $\left\{ T_{x}a\colon x\in\rd\right\} $ is relatively compact
in $C^{\infty}\left(\rd\right)$ by Lemma \ref{prop:Arzela}, and
Proposition \ref{strongcomp} says that the mapping $a\mapsto H_{a}f$
is continuous from \linebreak[3] $\overline{\left\{ T_{x}a\colon x\in\rd\right\} }^{C^{\infty}\left(\rd\right)}$
to $\sobolev s$. 
\end{proof}

\subsubsection{Slowly oscillating symbols.\label{subsec:Slowly-oscillating-symbols.}}

In the next step we single out a subclass of operators for which the
spectral theory is sufficiently simple. In our approach it is essential
that the limit operators do not have the endpoint $0$ and $\Omega$
of the spectrum as eigenvalues. 
The limits of translates of \emph{slowly oscillating} symbols are
constant, if they exist (Lemma \ref{prop:constlim} below), so the
limit operators are similar to the Laplacian. This will be used in
Section \ref{sec:Geometric-Beurling-Densities} to compute the critical
density. 

\begin{defn}
\label{def:slow}
An $X$-valued function $f\in C_{b}^{u}\left(\rd,X\right)$ is slowly
oscillating~\footnote{In the literature $f$ is also called ``of vanishing oscillation
at infinity\textquotedblright{} or a Higson function}, if for all compact subsets $M\subset\rd$ 
\[
\lim_{\abs x\to\infty}\sup_{m\in M}\norm[X]{f(x)-f(x+m)}=0\,.
\]
{} In fact, it suffices to use the closed unit ball $\overline{B_{1}}$
instead of arbitrary compact $M$.

We denote the space of all slowly oscillating functions on $\rd$
by $\sox X$ and define $\soix X=\sox X\cap\cbinfx X$.

The space $\so$ with the $\norm[\infty]{\phantom{.}}$-norm and pointwise
multiplication is a commutative $C^{*}$-subalgebra of $C_{b}^{u}\left(\rd\right)$
.

We will need the following characterization of $\soix X$. Though
the statement is folklore, we do not know a formal reference. For
completeness we sketch the simple proof. 
\end{defn}

\begin{lem}
\label{lem:charsoi}A function $f$ is in $\soix X$ if and only if
$f\in\cbinfx X$ and \linebreak[4] $\lim_{\abs x\to\infty}\partial_{k}f\left(x\right)=0$
for all $1\le k\le d$. 
\end{lem}

\begin{proof}
Assume that $f\in\cbinfx X$ and $\lim_{\abs x\to\infty}\partial_{k}f\left(x\right)=0$
for all $1\le k\le d$ and choose $M=[-h,h]^{d}.$ Writing $m\in M$
as $m=\sum_{k=1}^{d}h_{k}e_{k}$ with $|h_{k}|\leq h$, the difference
in Definition \ref{def:slow} is 
\[
f\left(x+\sum_{k=1}^{d}h_{k}e_{k}\right)-f(x)=\sum_{k=0}^{d-1}\int_{x+\sum_{l\leq k}h_{l}e_{k}}^{x+\sum_{l\leq k+1}h_{l}e_{k}}\partial_{k+1}f\,.
\]
This implies that $\sup_{m\in M}\norm[X]{f\left(x+m\right)-f(x)}\to0$
for $|x|\to\infty$.

Conversely, assume that $f\in\soix X$. If $\lim_{\abs x\to\infty}\partial_{k}f\left(x\right)\neq0$
for some $k$, then there exist $K>0$ and sequence $\left(x_{j}\right)_{j\in\n}\subset\rd$,
$\abs{x_{j}}\to\infty$, such that $\left\Vert \partial_{k}f\left(x_{j}\right)\right\Vert _{X}>K$,
and so there is a ball $B$ such that $\|\partial_{k}f\left(x\right)\|>K/2$
for $x\in x_{j}+B$. Then $\left\Vert f(x_{j}+\delta e_{k})-f(x_{j})\right\Vert _{X}=\delta\left\Vert \partial_{k}f\left(\xi\right)\right\Vert \geq\delta K/2$
for some $\xi\in x_{j}+B$. This contradicts $f\in\sox X$. 
\end{proof}

\begin{example}
 A typical  example of a genuinely slowly oscillating function is $a(x) =
\sin |x|^{1/2} (1-\phi (x))$ for some $\phi \in C^\infty _c(\rd )$ with
$\phi (x) = 1$ near $0$. (The cut-off of the singularity at $0$ serves
to make all derivatives of $a$ bounded, but, of course, it is
immaterial for the asymptotic behavior.) 
\end{example}
Our interest in $\sox X$ stems from the following fact (cf.~\cite[Prop.~2.4.1]{MR2075882}): 
\begin{lem}
\label{prop:constlim} Assume $f\in\sox X$ and $\left(x_{\lambda}\right)_{\lambda\in\Lambda}\subset\rd$
diverges to infinity, $\abs{x_{\lambda}}\to\infty$. If $\lim_{\lambda}T_{-x_{\lambda}}f(x)=g(x)$
exists for all $x\in\rd$, then g is constant. 
\end{lem}

\begin{proof}
Let $x,x'\in\rd$. Definition \ref{def:slow} with $M=\left\{ x,x'\right\} $
shows that for all $\varepsilon>0$ there exists an index $\lambda_{\varepsilon}=\lambda_{\varepsilon}(x,x')$
such that $\left\Vert f\left(x+x_{\lambda}\right)-f(x_{\lambda})\right\Vert _{X}<\varepsilon/2$
and $\left\Vert f\left(x'+x_{\lambda}\right)-f(x_{\lambda})\right\Vert _{X}<\varepsilon/2$
for all $\lambda\succeq\lambda_{\varepsilon}$. So $\left\Vert f\left(x+x_{\lambda}\right)-f(x'+x_{\lambda})\right\Vert _{X}<\varepsilon$
for all $\lambda\succeq\lambda_{\varepsilon}$. If $g=\lim_{\lambda}T_{-x_{\lambda}}f$
exists, it follows that $\left\Vert g\left(x\right)-g\left(x'\right)\right\Vert _{X}\leq\varepsilon.$
As $\epsilon>0$ was arbitrary, $g$ must be constant. 
\end{proof}

\section{Statement of the Density Theorem}

We state our main theorems. A first version describes a general setup
for symbols in the class $\cbinfd$ under additional assumptions on
the spectra of the limit operators. We then formulate a corollary
for slowly oscillating symbols, where the assumptions on the limit
operators are automatically satisfied. We discuss possible applications
of the general version in Section \ref{sec:Outlook}. 
\begin{thm}
\label{mainabstr} Assume that $H_{a}=-\sum_{j,k=1}^{d}\partial_{j}a_{jk}\partial_{k}$
is uniformly elliptic with symbol $a\in\cbinfd$. Let $\PW\left(H_{a}\right)$
be the Paley-Wiener space as defined in Section \ref{subsec:The-generalized-Paley}.
Assume that 
$\Omega$ is not an eigenvalue of any limit operator $H_{b}$ .

If $S$ is a set of stable sampling for $\PW\left(H_{a}\right)$,
then 
\begin{equation}
D_{0}^{-}\left(S\right)\geq1.\label{eq:5a-1-1-1}
\end{equation}
If $S$ is a set of interpolation for $\PW\left(H_{a}\right)$, then
\begin{equation}
D_{0}^{+}\left(S\right)\leq1\,.\label{eq:5b-1-1-1}
\end{equation}
\end{thm}

The following corollary is Theorem \nameref{tm2} of the introduction,
where we have used the equivalence of Lemma \ref{lem:charsoi} to
avoid the formal definition of $\soi$. 
\begin{cor}
\label{main} Assume that $H_{a}=-\sum_{j,k=1}^{d}\partial_{j}a_{jk}\partial_{k}$
is uniformly elliptic with symbol $a\in\soid$.

If $S$ is a set of stable sampling for $\PW\left(H_{a}\right)$,
then $D_{0}^{-}\left(S\right)\geq1.$

If $S$ is a set of interpolation for $\PW\left(H_{a}\right)$, then
$D_{0}^{+}\left(S\right)\leq1\,.$ 
\end{cor}

\begin{proof}[Proof of Corollary \ref{main}]
If $a\in\soid$, then by Lemma \ref{prop:Arzela} every net $\left(x_{\lambda}\right)_{\lambda\in\Lambda}\subset\rd$
that diverges to infinity has a subnet $\left(x_{\mu}\right)_{\mu\in M}$,
such that $\lim_{\mu}T_{-x_{\mu}}a=b$ in the topology of $C^{\infty}\left(\rd\right)$
for a symbol $b$. This symbol $b$ is constant by Lemma
\ref{prop:constlim} and positive definite; so $H_{b}$ is similar
to the Laplacian and has no point spectrum. 
\end{proof}

\section{Proof of Weak Localization of the kernel~\label{sec:Proof-of-WL}}

To prove Theorem~\ref{mainabstr} we invoke Theorem~\ref{abstrdens}
and verify its main hypotheses (WL) and (HAP) on the reproducing kernel.

Let $Q_{h}(x)=\left[x-\frac{h}{2};x+\frac{h}{2}\right]$ be the cube
of side-length $h$ at $x\in\rd$, and let $\varphi_{x}^{h}=\left|h\right|^{-d}\chi_{Q_{h}(x)}=T_{x}\varphi_{0}^{h}$
be the usual approximate unit. 
\begin{lem}
\label{lem:-uniformly-in}$\lim_{h\to0}\norm{\specproj{H_{a}}\varphi_{x}^{h}-k_{x}}=0$
uniformly in $x\in\rd$. 
\end{lem}

\begin{proof}
Let $f\in\PW(H_{a})$, then 
\begin{align*}
|\langle f,\specproj{H_{a}}\varphi_{x}^{h}-k_{x}\rangle| & =|\langle f,\varphi_{x}^{h}\rangle-f(x)|\\
 & =h^{-d}|\int_{Q_{h}(x)}(f(y)-f(x))\,dy|\leq h^{-d}\int_{Q_{h}(x)}|f(y)-f(x)|\,dy\\
 & \leq\sup_{z\in\rd}|\nabla f(z)|\,h^{-d}\int_{Q_{h}(x)}|y-x|\,dy\\
 & \leq C\|\nabla f\|_{\infty}\,\,h\,.
\end{align*}
Since $f\in\sobolev s(\rd)$ for all $s\geq0$, we apply first the
Sobolev embedding (with $s>d/2+1$) and then Lemma~\ref{lem:Sobolev}
and obtain 
\[
\|\nabla f\|_{\infty}\leq C_{1}\|f\|_{W_{2}^{s}}\leq C\|f\|_{2}\,,
\]
since $f\in\PW(H_{a})$. Consequently, 
\[
|\langle f,\specproj{H_{a}}\varphi_{x}^{h}-k_{x}\rangle|\leq Ch\|f\|_{2}\,.
\]
Taking the supremum over $f\in\PW(H_{a})$, we obtain 
\[
\|\specproj{H_{a}}\varphi_{x}^{h}-k_{x}\|_{2}=\sup_{f\in\PW(H_{a}),\|f\|_{2}=1}\langle f,\specproj{H_{a}}\varphi_{x}^{h}-k_{x}\rangle\leq Ch\,.
\]
As this estimate is independent of $x$, we have shown that $\specproj{H_{a}}\varphi_{x}^{h}\to k_{x}$
in $\lrd$ uniformly in $x$. 
\end{proof}
The following result relates the reproducing kernel of a limit operator
of $H_{a}$ to the original kernel. It expresses a form of continuous
dependence of the reproducing kernel of the matrix symbol of $H_{a}$.
We will denote the point spectrum of an operator $H$ by $\sigma_{p}(H)$. 
\begin{thm}
\label{rkconv-1} Let $H_{a}$ with symbol $a\in\cbinfd$, and let
$\left(x_{\lambda}\right)_{\lambda\in\Lambda}\subset\rd$ be an unbounded
net such that $\lim_{\lambda}T_{-x_{\lambda}}a=b$ pointwise. Assume
that 
\ $\Omega\notin\sigma_{p}\left(H_{b}\right)$. Let $\tilde{k}$ be
the reproducing kernel of $\PW(H_{b})$. Then 
\[
\lim_{\lambda}T_{-x_{\lambda}}k_{x_{\lambda}}=\tilde{k}_{0}\,
\]
with convergence in $\sobolev s$ for every $s\geq0$. 
\end{thm}

Before the proof we remind the reader of the following standard facts
of spectral theory (see, e.g.,~\cite[Ch.~6.6]{Teschl09}. Although in
the literature 
these results are formulated  for sequences of operators,  the statements and proofs
are  equally valid for nets\footnote{The cited results use the  the
  strong operator topology. As this topology  is metrizable on bounded
  sets, the convergence of nets is equivalent to the convergence of sequences.}.  Let $H_{\lambda},\lambda\in\Lambda$, and $H_{b}$ be self-adjoint
operators with a common core $\mathcal{D}$. If $H_{\lambda}f\to H_{b}f$
for all $f\in\mathcal{D}$, then, for every $F\in C_{b}(\r),$ 
\begin{equation}
F\left(H_{\lambda}\right)f\to F\left(H_{b}\right)f\qquad\text{for all }f\in\ll(\rd)\,.\label{eq:spprconv}
\end{equation}
Furthermore, if $\chi_{\left\{ \alpha\right\} }\left(H_{b}\right)=\chi_{\left\{ \beta\right\} }\left(H_{b}\right)=0,$
then 
\begin{equation}
\chi_{[\alpha,\beta]}\left(H_{\lambda}\right)f\to\chi_{[\alpha,\beta]}\left(H_{b}\right)f\qquad\text{for all }f\in\ll(\rd)\,.\label{eq:c17-1}
\end{equation}

\begin{proof}[Proof of Theorem~\ref{rkconv-1}]
We split the difference $T_{-x_{\lambda}}k_{x_{\lambda}}-\tilde{k}_{0}$
into three terms and then estimate their $W_{2}^{s}$-norm separately.
\begin{align}
\norm[\sobolev s]{T_{-x_{\lambda}}k_{x_{\lambda}}-\tilde{k}_{0}} & \leq\norm[\sobolev s]{T_{-x_{\lambda}}k_{x_{\lambda}}-T_{-x_{\lambda}}\specproj{H_{a}}\varphi_{x_{\lambda}}^{h}}+\norm[\sobolev s]{T_{-x_{\lambda}}\specproj{H_{a}}\varphi_{x_{\lambda}}^{h}-\specproj{H_{b}}\varphi_{0}^{h}}\nonumber \\
 & +\norm[\sobolev s]{\specproj{H_{b}}\varphi_{0}^{h}-\tilde{k}_{0}}=(I)+(II)+(III)\,.\label{eq:c20}
\end{align}
Choose $\varepsilon>0$.

\textbf{Step 1: }Expression \emph{(I)} can be estimated by 
\[
\norm[\sobolev s]{T_{-x_{\lambda}}k_{x_{\lambda}}-T_{-x_{\lambda}}\specproj{H_{a}}\varphi_{x_{\lambda}}^{h}}=\norm[\sobolev s]{k_{x_{\lambda}}-\specproj{H_{a}}\varphi_{x_{\lambda}}^{h}}\leq C_{s}\norm[2]{k_{x_{\lambda}}-\specproj{H_{a}}\varphi_{x_{\lambda}}^{h}}\,.
\]
The first equality holds by the translation invariance of the Sobolev
norm, the second inequality is a consequence of Proposition \ref{lem:Sobolev}.
By Lemma \ref{lem:-uniformly-in} there exists $h_{\varepsilon}>0$
such that, for every $0<h<h_{\varepsilon}$, 
\begin{equation}
\norm[2]{k_{x}-\specproj{H_{a}}\varphi_{x}^{h}}<\frac{\varepsilon}{3C_{s}}\label{eq:uc-1}
\end{equation}
for \emph{all} $x\in\rd$. So for $h<h_{\varepsilon}$, we obtain
$(I)<\varepsilon/3$. Similarly, we achieve \emph{(III)} $<\varepsilon/3$
for every $h<h_{\epsilon}'$.

\textbf{Step 2:} To bound the decisive term \emph{(II)}, we bring
in limit operators and elliptic regularity theory. Set $a_{\lambda}=T_{-x_{\lambda}}a$.
First note that 
\[
T_{-x_{\lambda}}\specproj{H_{a}}\varphi_{x_{\lambda}}^{h}=T_{-x_{\lambda}}\specproj{H_{a}}T_{x_{\lambda}}\varphi_{0}^{h}=\specproj{\tau_{x_{\lambda}}H_{a}}\varphi_{0}^{h}=\specproj{H_{a_{\lambda}}}\varphi_{0}^{h}\,.
\]
We have to verify that 
\begin{equation}
\lim_{\lambda}\norm[\sobolev s]{\specproj{H_{a_{\lambda}}}\varphi_{0}^{h}-\specproj{H_{b}}\varphi_{0}^{h}}=0.\label{eq:sobconv}
\end{equation}

For $L^{2}$-convergence ($s=0$) we argue as follows. By Lemma \ref{prop:Arzela}
the translates $T_{-x_{\lambda}}a$ converge to the matrix $b$ uniformly
on compact sets. Proposition \ref{strongcomp} implies that $H_{T_{-x_{\lambda}}a}f\to H_{b}f$
for $f\in\sobolev s$, $s\geq0$. To apply \eqref{eq:c17-1}, we note
that $\cinf$ is a common core for all $H_{a_{\lambda}}$ and for
$H_{b}$ and that $0\notin\sigma_{p}\left(H_{b}\right)$ by \eqref{eq:notev}
and $\Omega\notin\sigma_{p}\left(H_{b}\right)$ by assumption. Therefore
\eqref{eq:sobconv} follows from \eqref{eq:c17-1}.

For the convergence of \eqref{eq:sobconv} in general Sobolev spaces
$\sobolev s$ it suffices to treat the case $s=2k$ for every integer
$k$. Recall that by the results on elliptic regularity in Section~\ref{subsec:The-generalized-Paley}
the operator $\left(H_{a}^{k}+c_{k}\right)$ defines an isomorphism
from $W_{2}^{2k}(\rd)$ to $\Lii$, and since $\tau_{x_{\lambda}}\left(H_{a}^{k}+c_{k}\right)=H_{a_{\lambda}}^{k}+c_{k}$
we obtain 
\[
\norm[\sobolev{2k}\to L^{2}]{H_{a_{\lambda}}^{k}+c_{k}}=\norm[\sobolev{2k}\to L^{2}]{H_{a}^{k}+c_{k}}<\infty\,.
\]
The Sobolev norm can be estimated by the $L^{2}$-norm 
\[
\norm[\sobolev{2k}]{f}=\norm[\sobolev{2k}]{T_{x}f}\leq C_{s}\norm{(H_{a}^{k}+c_{k})T_{x}f}=C_{s}\norm{T_{-x}(H_{a}^{k}+c_{k})T_{x}f}
\]
independently of $x\in\rd$ . Thus $(II)$ can be estimated by the
$L^{2}$-norm, namely 
\begin{align}
\norm[\sobolev s]{\specproj{H_{a_{\lambda}}}\varphi_{0}^{h}-\specproj{H_{b}}\varphi_{0}^{h}} & \leq C_{s}\norm[2]{\left(H_{a_{\lambda}}^{k}+c_{k}\right)\specproj{H_{a_{\lambda}}}\varphi_{0}^{h}-\left(H_{a_{\lambda}}^{k}+c_{k}\right)\specproj{H_{b}}\varphi_{0}^{h}}\label{eq:sh3}\\
 & \hspace{-15mm}\leq C_{s}\|\left(H_{a_{\lambda}}^{k}+c_{k}\right)\specproj{H_{a_{\lambda}}}\varphi_{0}^{h}-(H_{b}^{k}+c_{k})\specproj{H_{b}}\varphi_{0}^{h}\|_{2}\nonumber \\
 & \hspace{-15mm}+C_{s}\|(H_{b}^{k}+c_{k})\specproj{H_{b}}\varphi_{0}^{h}-\left(H_{a_{\lambda}}^{k}+c_{k}\right)\specproj{H_{b}}\varphi_{0}^{h}\|_{2}\\
 & \hspace{-15mm}=A_{\lambda}+B_{\lambda}\,.
\end{align}
By Proposition \ref{strongcomp} we have $\left(H_{a_{\lambda}}^{k}+c_{k}\right)f\to(H_{b}^{k}+c_{k})f$
in $L^{2}$-norm for all $f\in\sobolev{2k}$. In particular, this
holds for $f=\specproj{H_{b}}\varphi_{0}^{h}$, thus $\lim_{\lambda}B_{\lambda}=0.$

For the first term we use spectral theory again. Define $F\in C_{c}\left(\r\right)$
such that its restriction to $[0,\Omega]$ satisfies 
\[
F(t)=t^{k}+c_{k}\qquad\text{ for }t\in[0,\Omega]\,.
\]
Then $F\left(t\right)\specproj t=(t^{k}+c_{k})\chi_{\Omega}(t)$,
and $\lim_{\lambda}F\left(\auto{x_{\lambda}}{H_{a}}\right)f=F\left(H_{b}\right)f$
for all $f\in\Lii$ by \eqref{eq:spprconv}. Since the product of
bounded operators is continuous in the strong operator topology, it
follows that 
\begin{align*}
\lim_{\lambda}(H_{a_{\lambda}}^{k}+c_{k})\specproj{H_{a_{\lambda}}}\varphi_{0}^{h} & =\lim_{\lambda}F(H_{a_{\lambda}})\left(\lim_{\lambda}\specproj{H_{a_{\lambda}}}\varphi_{0}^{h}\right)\\
 & =F(H_{b})\specproj{H_{b}}\varphi_{0}^{h}=(H_{b}^{k}+c_{k})\specproj{H_{b}}\varphi_{0}^{h}\,,
\end{align*}
and so $\lim_{\lambda}A_{\lambda}=0.$

We can finish the proof as follows. We have already chosen $h<\min\left\{ h_{\varepsilon},h'_{\varepsilon}\right\} $
so that the terms \emph{(I) }and \emph{(III)} are $<\epsilon/3$ for
all $\lambda\in\Lambda$. For this fixed $h>0$ we can find an index
$\lambda_{0}$ such that 
\begin{equation}
(II)\leq C\norm{\left(H_{a_{\lambda}}^{k}+c_{k}\right)\specproj{H_{a_{\lambda}}}\varphi_{0}^{h}-\left(H_{a_{\lambda}}^{k}+c_{k}\right)\specproj{H_{b}}\varphi_{0}^{h}}<\frac{\varepsilon}{3}\label{eq:ab-1}
\end{equation}
for all $\lambda\succ\lambda_{0}$. Altogether we obtain $\norm{T_{-x_{\lambda}}k_{x_{\lambda}}-\tilde{k}_{0}}\leq(I)+(II)+(III)<\varepsilon\,.$ 
\end{proof}
\begin{thm}
\label{thm: L2relcpt} Assume that $H_{a}$ is uniformly elliptic
with symbol $a\in\cbinfd$ and that no limit operator has the eigenvalue
$\Omega$. Then the set $\left\{ T_{-x}k_{x}\colon x\in\rd\right\} $
is relatively compact in $\sobolev s$ for every $s\geq0$. 
\end{thm}

\begin{proof}
This follows directly from Theorem \ref{rkconv-1}. Let $(x_{n})_{n\in\n}\subseteq\rd$
be an arbitrary sequence. By Lemma \ref{prop:Arzela} the sequence
$T_{-x_{n}}a$ has a $C^{\infty}$-convergent subsequence
$T_{-x_{n_{l}}}a$. If $\left(x_{n_{l}}\right)_{l\in\n}$ is bounded we
can assume without 
loss of generality that $x_{n_{l}}\to x\in\rd$, and $T_{-x_{n_{l}}}k_{x_{n_{l}}}\to T_{-x}k_{x}$
in $\sobolev s$ by the continuity of the translations and Proposition
\ref{prop:kernbound}. If $(x_{n_{l}})_{l\in\n}$ is unbounded
we can assume $\abs{x_{n_{l}}}\to\infty$. This case is settled by
Theorem \ref{rkconv-1} and yields the convergence of $T_{-x_{n_{l}}}k_{x_{n_{l}}}.$ 
\end{proof}
A combination of the above arguments yields the weak localization
\eqref{eq:hapcont} . 
\begin{thm}
\label{wla} Assume that $H_{a}$ is uniformly elliptic with symbol
$a\in\cbinfd$ and that no limit operator has the eigenvalue 
$\Omega$. 
Let $k$ be the reproducing kernel of $\PW(H_{a})$. Then $k$ satisfies
the weak localization property \eqref{eq:hapcont}, i.e., 
\[
\lim_{R\to\infty}\int_{\abs{y-x}>R}\abs{k\left(x,y\right)}^{2}dy=0\,.
\]
\end{thm}

\begin{proof}
By Theorem \ref{thm: L2relcpt} (for $s=0)$ the set $\left\{ T_{-x}k_{x}\colon x\in\rd\right\} $
is relatively compact in $\Lii$ . The Riesz-Kolmogorov theorem implies
that for all $\varepsilon>0$ there is $R>0$ such that for all $x\in\rd$
\[
\int_{\rd\setminus B_{R}\left(0\right)}\abs{T_{-x}k_{x}\left(y\right)}^{2}dy<\varepsilon^{2}\,.
\]
By a change of variable this expression reads as 
\[
\int_{\abs{y-x}>R}\abs{k\left(x,y\right)}^{2}dy<\varepsilon^{2}\,,
\]
and this is \eqref{eq:hapcont}. 
\end{proof}

\section{Proof of the Homogeneous Approximation Property (HAP) \label{sec:Proof-of-HAP}}

Next we prove the homogeneous approximation property. Recall that
$T_{x}\kappa$ is the reproducing kernel for $\sobolev s$ with $\hat{\kappa}(\omega)=(1+|\omega|^{2})^{-s}$. 
\begin{lem}
\label{lem:sobolevBessel}If $S$ is a relatively separated set in
$\rd$, then $\left\{ T_{x}\kappa\colon x\in S\right\} $ is a Bessel
sequence for $\sobolev s$, $s>d/2$. 
\end{lem}

\begin{proof}
By standard facts of frame theory (see, e.g. \cite[Thm 7.6]{Heil11})
the Bessel property is equivalent to the boundedness of the Gramian
$G=\big(\inprod{T_{x}\kappa,T_{y}\kappa}_{\sobolev s}\big)_{x,y\in S}$
on $\ell^{2}\left(S\right).$ 
To deduce the boundedness of $G$ we first show that $G$ possesses
exponential off-diagonal decay and then apply Schur's test. The off-diagonal
decay follows from a (well-known) calculation, see, e,g. \cite[Thm 6.13]{Wendland05}
or~\cite[App.~B]{grafakos}. Let $\mathscr{J}_{r}$ denote the Bessel
function of the first kind and $\mathscr{K}_{r}$ the modified Bessel
function of the second kind. Then 
\begin{align}
\inprod{T_{x}\kappa,T_{y}\kappa}_{\sobolev s} & =\left(2\pi\right)^{-d/2}\int_{\rd}\widehat{T_{x}\kappa\,}\left(\omega\right)\overline{\widehat{T_{y}\kappa}}\left(\omega\right)(1+\abs{\omega}^{2})^{s}d\omega\nonumber \\
 & =\left(2\pi\right)^{-d/2}\int_{\rd}e^{-i\left(x-y\right)\omega}(1+\abs{\omega}^{2})^{-s}d\omega\nonumber \\
 & =C\abs{x-y}^{-(d-2)/2}\int_{0}^{\infty}\left(1+r^{2}\right)^{-s}\mathscr{J}_{\frac{d-2}{2}}\left(r\abs{x-y}\right)r^{d/2}dr\label{eq:exdec}\\
 & =C'\abs{x-y}^{s-d/2}\mathscr{K}_{s-d/2}\left(\abs{x-y}\right)\,.\nonumber 
\end{align}
Using the asymptotic decay $\mathscr{K}_{r}(x)\sim\sqrt{\pi/(2x)}e^{-x}$
for $x\to\infty$, see, e.g., \cite[Eq. 10.25.3]{DLMF}, the off-diagonal
decay of $G$ is 
\begin{equation}
\abs{\langle T_{x}\kappa,T_{y}\kappa\rangle_{\sobolev s}}\leq C''\abs{x-y}^{s-d/2-1/2}e^{-\abs{x-y}}\qquad(\abs{x-y}\to\infty)\,.\label{eq:c10}
\end{equation}
The off-diagonal decay of the Gramian implies the boundedness of the
Gramian as follows. By \eqref{eq:c10} there exists $N_{0}\in\n$
such that $\abs{G_{xy}}\leq Ce^{-c\abs{x-y}}$, if \textbf{$\abs{x-y}>N_{0}.$}

For $x\in S$ and $k\in\n$ set $A_{k}\left(x\right)=\left\{ y\in S\colon k<\abs{y-x}\leq k+1\right\} $.
Since $S\subset\rd$ is relatively separated, 
there exists $r>0$ such that 
\[
\max\#(S\cap B_{r}(x))<\infty\,.
\]
A covering argument (of a large ball $B_{R}(z)$ by balls $B_{r}(x)$)
implies that $\#(S\cap B_{R}(z))=\cO(R^{d})$ for arbitrary $R>0$.
Consequently we also obtain $\#\,\text{\ensuremath{A_{k}\left(x\right)}\ensuremath{\ensuremath{\leq}C\ensuremath{k^{d}}}}$
independent of $x$. Then 
\begin{align*}
\sum_{y\in S}\abs{G_{xy}} & =\sum_{k\in\n}\sum_{y\in A_{k}\left(x\right)}\abs{G_{xy}}\\
 & =\sum_{k\leq N_{0}}\sum_{y\in A_{k}\left(x\right)}\abs{G_{xy}}+\sum_{k>N_{0}}\sum_{y\in A_{k}\left(x\right)}\abs{G_{xy}}\\
 & \leq C\,\#\left(B_{N_{0}+1}\left(x\right)\cap S\right)+C\sum_{k>N_{0}}e^{-ck}\,\#A_{k}(x)\\
 & \leq C_{1}(N_{0}+1)^{d}+C_{2}\sum_{k>N_{0}}e^{-ck}k^{d}\,.
\end{align*}
This expression is bounded independently of $x$. 
Now Schur's test implies that the Gramian is bounded on $\ell^{2}\left(S\right)$. 
\end{proof}
\begin{thm}[HAP]
Assume that $H_{a}$ is uniformly elliptic with symbol $a\in\cbinfd$
and that 
$\Omega\notin\sigma_{p}\left(H_{b}\right)$ for every limit operator
$H_{b}$. Let $\left\{ k_{x}\colon x\in S\right\} $ be a Bessel sequence
in $PW_{\Omega}\left(H_{a}\right)$. Then for every $\varepsilon>0$
there exists an $R>0$ such that for all $y\in S$ 
\[
\sum_{x\in S\setminus B_{R}(y)}\abs{k(y,x)}^{2}<\varepsilon^{2}.
\]
\end{thm}

\begin{proof}
If $\left\{ k_{x}\colon x\in S\right\} $ is a Bessel sequence of
reproducing kernels, then $S$ is relatively separated in $\rd$ (see
\cite[Lemma 3.7]{densrkhs16}). Lemma~\ref{lem:sobolevBessel} implies
that $\{T_{x}\kappa:x\in S\}$ is also a Bessel sequence in $\sobolev s$
for $s>d/2$.

Choose $\varepsilon>0.$ Since $\left\{ T_{-x}k_{x}\colon x\in\rd\right\} $
is relatively compact in $\sobolev s$ for $s\geq0$ by Theorem \ref{thm: L2relcpt},
the Kolmogorov-Riesz theorem for translation-invariant Banach spaces~\cite{Fei84}
asserts that there exists a $R=R_{\varepsilon}>0$ and a function
$\psi\in C_{c}^{\infty}\left(\rd\right)$ satisfying $\psi|_{B_{R/2}(0)}=1$,
$\supp\psi\subseteq B_{R}(0)$ such that 
\[
\norm[\sobolev s]{T_{-x}k_{x}\left(1-\psi\right)}\leq\varepsilon\quad\text{for all \ensuremath{x\in\rd}\,.}
\]
We now use the fundamental observation \eqref{eq:c12-1} that the
point evaluation in $\PW(H_{a})$ can be expressed in two ways. For
$f=k_{x}$ we have 
\begin{equation}
k(y,x)=\langle k_{y},k_{x}\rangle_{L^{2}}=\langle k_{y},T_{x}\kappa\rangle_{\sobolev s}\,.\label{eq:c13}
\end{equation}
Since $\left\{ T_{x}\kappa\colon x\in S\right\} $ is a Bessel sequence
in $\sobolev s$ with bound $B$, the set $\left\{ T_{x-y}\kappa\colon x\in S\right\} $
is a Bessel sequence with the same bound. Observe that for $\abs u>R$
we obtain 
\[
\inprod{\psi T_{-y}k_{y},T_{u}\kappa}_{\sobolev s}=T_{-y}k_{y}(u)\psi(u)=0.
\]
This implies

\begin{align*}
\sum_{x\in S\setminus B_{R}(y)}\abs{\inprod{k_{y},T_{x}\kappa}_{\sobolev s}}^{2} & =\sum_{x\in S\setminus B_{R}(y)}\abs{\inprod{T_{-y}k_{y},T_{x-y}\kappa}_{\sobolev s}}^{2}\\
 & =\sum_{x\in S\setminus B_{R}(y)}\abs{\inprod{\left(1-\psi)\right)T_{-y}k_{y},T_{x-y}\kappa}_{\sobolev s}}^{2}\\
 & \leq B\norm[\sobolev s]{\left(1-\psi)\right)T_{-y}k_{y}}^{2}\leq B\varepsilon^{2}\,,
\end{align*}
and this is HAP. 
\end{proof}
\begin{proof}[Proof of Theorem~\ref{mainabstr}]
After the verification of the properties (\ref{eq:hapcont}) and
(\ref{eq:hapclass}) of the reproducing kernel, the version for the
dimension-free density of Theorem~\ref{abstrdens} is applicable
and yields Theorem~\ref{mainabstr}. The statement asserts the existence
of a critical density for sampling and interpolation with the dimension-free
Beurling density $D_{0}^{\pm}(S)$. 
\end{proof}

\section{Geometric Beurling Densities\label{sec:Geometric-Beurling-Densities}}

In this section we derive results for the geometric densities~\eqref{eq:2low}.
According to 
Theorem~\ref{abstrdens} this step requires the computation of the
averaged trace $|\mu(B_{r}(x))|^{-1}\int_{B_{r}(x)}k(x,x)\,d\mu(x)$
of the reproducing kernel. This version of the density theorems is
of interest because it relates the critical density in $\PW(H_{a})$
to the geometry defined by the differential operator $H_{a}$. The
explicit computation of the averaged trace becomes possible by introducing
a suitable compactification of $\rd$ and then extending the centered
kernels $T_{-x}k_{x}$ to this compactification.

\subsection{The basic computation: constant coefficients\label{subsec:Constant-Coefficients}}

For reference we mention the case when $H_{b}=-\sum_{j,k}\partial_{j}b_{jk}\partial_{k}=-\sum_{j,k}b_{jk}\partial_{j}\partial_{k}$
is a differential operator with  constant coefficients
$b_{jk}$. 

Define 
\[
\Sigma_{\Omega}^{b}=\{\xi\in\rd:b\xi\cdot\xi\leq\Omega\}=b^{-1/2}\overline{B_{\Omega^{1/2}}(0)}\,
\]
with volume 
\begin{equation}
\abs{\Sigma_{\Omega}^{b}}=\det\big(b^{-1/2}\big)\abs{B_{\Omega^{1/2}}(0)}=\det\big(b^{-1/2}\big)\Omega^{d/2}\abs{B_{1}}\,.\label{eq:volepps}
\end{equation}
Since $\widehat{H_{b}f}(\xi)=\sum_{j,k}b_{jk}\xi_{j}\xi_{k}\hat{f}(\xi)=(b\xi\cdot\xi)\hat{f}(\xi)$,
the spectral subspace is 
\[
\PW(H_{b})=\chi_{[0,\Omega]}(H_{b})\Lii=\{f\in\Lii\colon\supp\hat{f}\subseteq\Sigma_{\Omega}^{b}\}\,.
\]
The kernel of $\PW(H_{b})$ is 
\begin{equation}
\tilde{k}(x,y)=(2\pi)^{-d/2}(\mathcal{F}^{-1}\chi_{\Sigma_{\Omega}^{b}})(x-y)\,,\label{eq:pwkern}
\end{equation}
whence 
\begin{equation}
\tilde{k}(x,x)=(2\pi)^{-d/2}(\mathcal{F}^{-1}\chi_{\Sigma_{\Omega}^{b}})(0)=\frac{\abs{\Sigma_{\Omega}^{b}}}{\left(2\pi\right)^{d}}=\frac{\abs{B_{1}}}{\left(2\pi\right)^{d}}\det(b^{-1/2})\Omega^{d/2}\,.\label{eq:limkern}
\end{equation}
By Landau's theorem~\cite{landau67} a sampling set $S$ for $\PW(H_{b})$
has Beurling density at least $D^{-}(S)\geq|\Sigma_{\Omega}^{b}|/(2\pi)^{d}$.

\subsection{The Higson compactification\label{subsec:The-Higson-compactification}}

We recall how a compactification arises in  Gelfand theory. Let
$C_{\gamma}\left(\rd\right)$ be a  unital
$C^{*}$ algebra of functions on $\rd $ satisfying the embeddings
$C_{0}\left(\rd\right)\subset C_{\gamma}\left(\rd\right)\subset
C_{b}\left(\rd\right)$. 
The \emph{maximal ideal space }$M_{\gamma}$ of
$C_{\gamma}\left(\rd\right)$ is the space of all multiplicative homomorphisms
$\phi\colon C_{\gamma}\left(\rd\right)\to\c$. Equipped with the weak-star
topology $M_{\gamma}$ is a compact Hausdorff space. The point evaluations
$\delta_{x}\left(f\right)=f\left(x\right)$ constitute an embedding
$\gamma$ of $\rd$ into $M_{\gamma}$ via $\gamma\left(x\right)=\delta_{x}$,
and $\gamma\left(\rd\right)$ is dense in $M_{\gamma}$ and homeomorphic
to $\rd$. Thus, the pair $\left(\gamma,M_{\gamma}\right)$ is a compactification
of $\rd$, which we will call $\gamma\rd$. The corona of $\gamma\rd$
is $\partial_{\gamma}\rd=\gamma\rd\setminus\gamma\left(\rd\right)$.
By abuse of notation we will  identify a point  $x\in \rd$ with
its point evaluation $\delta _x$. Then $C_{\gamma}\left(\rd\right)$ is isometrically
isomophic to $C\left(\gamma\rd\right)$. We denote the image of $f\in C_{\gamma}\left(\rd\right)$
in $C\left(\gamma\rd\right)$ by $\bar{f}$. See, e.g., \cite{Engelking89}
for the basics of compactifications, and \cite{Gamelin69} for compactifications
of function algebras.

As noted in Section \ref{subsec:Slowly-oscillating-symbols.} the
space $\so$ of slowly oscillating functions with supremum norm is
a commutative unital $C^{*}$-algebra. Thus there is a compactification
$h\rd$ of $\rd$, the \emph{Higson compactification}, such that $\so$
is isometrically isomorphic to $C\left(h\rd\right)$. It is known
that $h\rd$ is non-metrizable, and even more,  points of the corona $h\rd\setminus\rd$
can only be reached by nets, see, e.g., \cite[2.4.10]{MR2075882}. 
Therefore we need to work with nets instead of sequences.

The relevance of the Higson compactification and the algebra of slowly
oscillating functions in our context is given by the fact that translations
act trivially on the corona $\partial_{h}\rd$.
\begin{lem}
\label{lem:trivtrans-1} If $x_{\lambda}\to\eta\in\partial_{h}\rd$, 
then $x+x_{\lambda}\to\eta$ for all $x\in\rd$.
\end{lem}

\begin{proof}
By definition $x_{\lambda}\to\eta\in\partial_{h}\rd$ if $f\left(x_{\lambda}\right)\to\bar{f}\left(\eta\right)=\eta\left(f\right)$
for every $f\in\so$. From the definition of $\so$ we obtain \label{lem:trivtrans}
\[
\lim_{\lambda}\abs{f\left(x_{\lambda}\right)-f\left(x+x_{\lambda}\right)}=0
\]
 for every $x\in\rd$, so $f\left(x_{\lambda}+x\right)\to\bar{f}\left(\eta\right)$
for every $f\in\so$ as well.
\end{proof}
One can show that $h\rd$ is the \emph{maximal} compactification of
$\rd$ with this property: every $C_{\gamma}\left(\rd\right)$ as
above with translations acting trivially on $\partial_{\gamma}\rd$
is a subalgebra of $\so$. 

We need the following fact \cite{Roe03}.
\begin{prop}
\label{prop:C0kernel}Let $C_{\gamma}\left(\rd\right)$ be a $C^{*}$-algebra
of functions on $\rd$ as above with corresponding compactification
$\gamma\rd$ of $\rd$. If $f\in C_{\gamma}\left(\rd\right)$ satisfies
\[
\bar{f}|_{\partial_{\gamma}\rd}\equiv0
\]
 then $f\in C_{0}\left(\rd\right)$.
\end{prop}

\begin{proof}
Let $(x_{\lambda})_{\lambda\in\Lambda}\subset\rd$ be an
unbounded net, $\lim_{\lambda}\abs{x_{\lambda}}=\infty$. As $\gamma\rd$
is compact, every subnet of $(x_{\lambda})$ has a convergent
subnet $(x_{\mu})_{\mu\in M}$, and $\lim_{\mu}x_{\mu}=\eta\in\partial_{\gamma}\rd$
by the assumption of unboundedness. So $\lim_{\mu}f(x_{\mu})=\bar{f}(\eta)=0$
for a subnet of a given subnet, and therefore $\lim_{\lambda}f(x_{\lambda})=0$.
This means $f\in C_{0}(\rd)$.
\end{proof}
We next study uniformly elliptic operators $H_{a}$ with a symbol
in $a\in C_{h}^{\infty}(\rd,\mathbb{C}^{d\times d})$.
By definition $a$ has a continuous extension to $h\rd$. Thus for
$\eta\in\partial_{h}\rd$ the symbol $\bar{a}\left(\eta\right)=\lim_{x\to\eta}a(x)$
is well defined, and by Proposition \ref{prop:constlim}
$H_{\bar{a}\left(\eta\right)}$ is a differential operator with
constant coefficients. Let $k^{\eta}$ denote the reproducing 
kernel of $\PW\left(H_{\bar{a}(\eta)}\right)$. We show that the
mapping $x\mapsto T_{-x}k_{x}$ has a continuous extension to $h\rd$. 
\begin{prop}
\label{prop: higs-cont}Let the symbol $a\in\soid$ be the symbol
of the operator $H_{a}$ and let $k_{x}$ be the reproducing kernel
of $\PW(H_{a})$. Set $K(x)=T_{-x}k_{x}\in\lrd$. Then $K$ extends
to a continuous function from $h\rd$ to $\lrd $ 
by setting $K\left(\eta\right)=k_{0}^{\eta}$ for $\eta\in h\rd$. In
particular, the diagonal $k(x,x) = \|k_x\|_2^2$ is a slowly
oscillating function. 
\end{prop}

\begin{proof}
(i) By Proposition \ref{cor:normcont} the centered reproducing kernel $K$  is continuous
from $\rd$ to $\Lii$. To show that $K\in C_h (\rd, \mathbb{C} ^{d
  \times d})$, we need to extend $K$ to the Higson corona $\partial
_h\rd $.

(ii) This is accomplished by means of the fundamental
Theorem~\ref{rkconv-1}. Let $(x_\lambda )\subseteq \rd $ be an unbounded
net  such that $x_\lambda \to \eta \in \partial _h \rd $. Since $a\in
C_h  (\rd , \mathbb{C} ^{d \times d})$, there is a continuous
function $\overline{a} \in C(h\rd , \mathbb{C} ^{d \times d})$, such that
$\lim _{\lambda  } a(x_\lambda) = \overline{a}( \eta )$. Furthermore,  for $x\in
\rd $ arbitrary, $x+x_\lambda \to \eta $ by Lemma~\ref{lem:trivtrans-1}, and this fact implies the
pointwise convergence 
$$
\lim _{\lambda } T_{-x_\lambda}a(x) = \overline{a} (\eta ) \, .
$$
Clearly, the spectrum of the (constant coefficient) operator
$H_{\overline{a} (\eta )}$ is continuous and does not contain any
eigenvalues. The assumptions of Theorem~\ref{rkconv-1} are thus
satisfied.

(iii) To formulate its conclusion, denote the reproducing kernel of
$\PW (H_{\overline{a} (\eta )})$ by $k^\eta $. Then by
Theorem~\ref{rkconv-1}
$$
\lim _{\lambda} T_{-x_\lambda}k_{x_\lambda} = k_0^\eta
$$
in the $L^2$-norm, and this holds for every net $(x_\lambda)$ with
$x_\lambda \to \eta $. Thus we must take the limiting function to
be
$$
K (\eta ) = k_0^\eta =
(2\pi)^{-d/2}
\mathcal{\mathcal{F}}^{-1}\Big(\chi_{\Sigma_{\Omega}^{\overline{a}
      \left(\eta\right)}}\Big) \, ,
$$
with the explicit formula for the kernel given by \eqref{eq:pwkern}. 

(iv) It remains to be shown that the limiting kernel $K $ is
continuous on $\partial _h \rd$. Let $\eta _\lambda  \to \eta \in
\partial _h \rd $. Then with
the definition of $\Sigma _\Omega ^b$ and \eqref{eq:volepps} we obtain 
\begin{align*}
\norm[2]{k_{0}^{\eta_{\lambda }}-k_{0}^{\eta}}^{2}
&=(2\pi )^{-d} \| \chi_{\Sigma_{\Omega}^{\overline{a} \left(\eta_{\lambda}\right)}}
  -\chi_{\Sigma_{\Omega}^{\overline{a}\left(\eta\right)}}\| _2^{2}\\
  &=(2\pi )^{-d} \Big( |\Sigma_{\Omega}^{\overline{a}
      \left(\eta_{\lambda}\right)}|+|\Sigma_{\Omega}^{\overline{a}
      \left(\eta\right)}|-2 |\Sigma_{\Omega}^{\overline{a}
      \left(\eta_{\lambda}\right)}\cap\Sigma_{\Omega}^{\overline{a}
      \left(\eta\right)}|\Big) \, .  
\end{align*}
As $a\in C_h (\rd , \mathbb{C}^{d  \times d})$, $\overline{a} $ is continuous
on $\partial _h \rd$, and this expression tends to $0$, whence $K $ is
also continuous on the corona  $\partial _n \rd $. 
\end{proof}


\subsection{Geometric densities for slowly oscillating symbols.}

In order to obtain values for the critical densities $D_{\mu}^{\pm}\left(S\right)$
we need the averaged traces 
\[
\tr_{\mu}^{-}\left(k\right)=\liminf_{r\to\infty}\inf_{x\in\rd}\frac{1}{\mu\left(B_{R}(x)\right)}\int_{B_{R}\left(x\right)}k\left(z,z\right)dz
\]
and $\tr_{\mu}^{+}\left(k\right)$. 

For the comparison of averaged traces we will need the following well-known
fact, whose proof is supplied in the appendix for completeness.

For $f\in C_{b}\left(\rd\right)$ set $\tr^{-}\left(f\right)=\liminf_{r\to\infty}\inf_{y\in\rd}\frac{1}{\abs{B_{r}\left(y\right)}}\int_{B_{r}\left(y\right)}f\left(x\right)\,dx$,
and define $\tr^{+}\left(f\right)$ similarly with $\sup$ instead
of $\inf$. 
\begin{lem}
\label{lem-sphint-1} Assume that $f,g\in C_{b}\left(\rd\right)$
and $\lim_{x\to\infty}|f(x)-g(x)|=0$. Then 
\[
\tr^{-}\left(f\right)=\tr^{-}\left(g\right)\quad\text{ and }\quad\tr^{+}\left(f\right)=\tr^{+}\left(g\right)\,.
\]
\end{lem}

\begin{prop}
If the symbol $a\in\soid$, then the trace of the reproducing kernel
satisfies 
\begin{equation}
\lim_{r\to\infty}\sup_{y\in\rd}\left|\frac{1}{\abs{B_{r}\left(y\right)}}\int_{B_{r}\left(y\right)}\left(k\left(x,x\right)-\frac{\left|B_{1}\right|}{\left(2\pi\right)^{d}}\frac{\Omega^{d/2}}{\det a(x)^{1/2}}\right)dx\right|=0\,.\label{eq:weyllike}
\end{equation}
Equivalently, using the Borel measure $\nu\left(B\right)=\int_{B}\det a(x)^{-1/2}dx$
\begin{equation}
\lim_{r\to\infty}\sup_{y\in\rd}\left|\frac{1}{\nu\left(B_{r}\left(y\right)\right)}\int_{B_{r}\left(y\right)}k\left(x,x\right)dx-\frac{\left|B_{1}\right|}{\left(2\pi\right)^{d}}\Omega^{d/2}\right|=0\,.\label{eq:wl2}
\end{equation}
Consequently, the averaged trace is 
\begin{equation}
\tr_{\nu}^{+}\left(k\right)=\tr_{\nu}^{-}\left(k\right)=\frac{\left|B_{1}\right|}{\left(2\pi\right)^{d}}\Omega^{d/2}\,.\label{eq:limev}
\end{equation}
\end{prop}

\begin{proof}
We apply Lemma~\ref{lem-sphint-1} to the functions $f(x)=k(x,x)$
and $g(x)=\frac{|B_{1}|\Omega^{d/2}}{(2\pi)^{d}}\,\det a(x)^{-1/2}$.
Then $k(x,x)$ is bounded by Proposition~\ref{prop:kernbound} and
continuous by Proposition~\ref{cor:normcont}. Likewise $\det a(x)^{-1/2}$
is bounded and continuous by elliptic regularity. By assumption on
$a$ and Proposition \ref{prop: higs-cont} both functions are in
$\so$ and thus possess the limits $\bar{a}\left(\eta\right)$ and
$\|K\left(\eta\right)\|_{2}^{2}$, in particular for $a$ this means
that 
\[
\lim_{x_{\lambda}\to\eta}\det a(x_{\lambda})^{-1/2}=\det\bar{a}(\eta)^{-1/2}\,.
\]
Using the notation of Section~\ref{subsec:Constant-Coefficients} and 
Proposition \ref{prop: higs-cont},  we obtain 
\begin{align*}
\lim_{x_{\lambda}\to\eta}\|k_{x_{\lambda}}\|_{2}^{2} & =\lim_{x_{\lambda}\to\eta}\|T_{-x_{\lambda}}k_{x_{\lambda}}\|_{2}^{2}\\
 & =\|k_{0}^{\eta}\|_{2}^{2}=|\Sigma_{\Omega}^{\bar{a}(\eta)}|\\
 & =\frac{|B_{1}|\Omega^{d/2}}{(2\pi)^{d}}\,\det\bar{a}(\eta)^{-1/2}\,.
\end{align*}
We conclude that both $f$ and $g$ have the same limit function,
and therefore $f-g\in C_{0}(\rd)$ by means of Proposition \ref{prop:C0kernel}.
Lemma~\ref{lem-sphint-1} now yields \eqref{eq:weyllike}. Equation
\eqref{eq:wl2} follows after multiplying with $|B_{r}(y)|/\nu(B_{r}(y))$
and taking limits. Finally, \eqref{eq:limev} is a direct consequence
of \ref{eq:wl2}. 
\end{proof}
Equation \eqref{eq:wl2} allows us to state our main result on geometric
Beurling densities for operators with slowly oscillating symbols
in a simple form. In order to do so we need an elementary result on
the relation between density and a change of measure. 
\begin{lem}
\label{lem:changemeas} Let $d\mu=hdx$ for a positive, continuous
function $h$ on $\rd$, bounded above and below, $0<c<h\left(z\right)<C$
for all $z\in\rd$. Then the dimension free density condition 
\[
D_{0}^{-}\left(S\right)=\liminf_{r\to\infty}\inf_{x\in\rd}\frac{\#\left(S\cap B_{r}\left(x\right)\right)}{\int_{B_{r}\left(x\right)}k\left(y,y\right)dy}\geq1
\]
holds, if and only if 
\[
D_{\mu}^{-}\left(S\right)\geq\tr_{\mu}^{-}\left(k\right).
\]
Similarly 
\[
D_{0}^{+}\left(S\right)\leq1\quad\text{if and only if }\quad D_{\mu}^{+}\left(S\right)\leq\tr_{\mu}^{-}\left(k\right)\,.
\]
\end{lem}


\begin{proof}
The inequality $D_{0}^{-}(S)\geq1$ means that for all $\varepsilon>0$
there is an $r_{\epsilon}>0$ such that for all $r>r_{\varepsilon}$
\begin{equation}
\#\left(S\cap B_{r}\left(x\right)\right)\geq\left(1-\varepsilon\right)\int_{B_{r}\left(x\right)}k\left(y,y\right)dy\,,\label{eq:denseps}
\end{equation}
or equivalently, 
\[
\frac{\#\left(S\cap B_{r}\left(x\right)\right)}{\mu\left(B_{r}\left(x\right)\right)}\geq\left(1-\varepsilon\right)\frac{1}{\mu\left(B_{r}\left(x\right)\right)}\int_{B_{r}\left(x\right)}k\left(y,y\right)dy\,.
\]
Written in terms of the Beurling density, this is 
\[
D_{\mu}^{-}\left(S\right)\geq\liminf_{r\to\infty}\inf_{x\in\rd}\frac{\int_{B_{r}\left(x\right)}k\left(y,y\right)dy}{\mu(B_{r}(x))}=\mathrm{tr}_{\mu}^{-}(k)\,.
\]
The converse is obtained by reading the argument backwards. 
\end{proof}
As a direct consequence we obtain the main result on geometric Beurling
densities for uniformly elliptic operators with slowly oscillating
symbols. This is Theorem \nameref{tm3} of the introduction. 
\begin{thm}
\label{mgen} Assume $H_{a}=-\sum_{j,k=1}^{d}\partial_{j}a_{jk}\partial_{k}$
is uniformly elliptic with symbol $a\in\soid$. Let $\PW\left(H_{a}\right)=\specproj{H_{a}}\lrd$
be the corresponding Paley-Wiener space and set $d\nu(x)=\det\left(a(x)\right)^{-1/2}dx$.

If $S\subseteq\rd$ is a set of stable sampling for $\PW\left(H_{a}\right)$
then 
\begin{equation}
D_{\nu}^{-}\left(S\right)\geq\frac{\left|B_{1}\right|}{\left(2\pi\right)^{d}}\Omega^{d/2}.\label{eq:ld}
\end{equation}
If $S\subseteq\rd$ is a set of interpolation for $\PW\left(H_{a}\right)$,
then 
\begin{equation}
D_{\nu}^{+}\left(S\right)\leq\frac{\left|B_{1}\right|}{\left(2\pi\right)^{d}}\Omega^{d/2}\,.\label{eq:ud}
\end{equation}
\end{thm}

\begin{proof}
We only verify the first assertion. By Theorem \ref{main}, if $S$
is a set of stable sampling, then $D_{0}^{-}\left(S\right)\geq1.$
By Lemma \ref{lem:changemeas} this is equivalent to 
\[
D_{\nu}^{-}\left(S\right)\geq\tr_{\nu}^{-}\left(k\right)\,.
\]
The averaged trace $\mathrm{tr}_{\nu}^{-}\,(k)$ was computed in~\eqref{eq:limev}
to be $\frac{\left|B_{1}\right|}{\left(2\pi\right)^{d}}\Omega^{d/2}$.
\end{proof}
\begin{example}
  We consider some special cases of Theorem~\ref{mgen}.

  (i) \emph{Asymptotically constant symbols.} Assume that
$a\in\cbinfd$ and that $\lim_{x\to\infty}a(x)=b$.  Then it is straightforward
to verify that $a\in\soid$ and that  $D_\nu ^{\pm} (S) = (\det b)
^{1/2} D^{\pm}(S)$.  Thus we may use the original Beurling density,
and Theorem~\ref{mgen} implies that  a sampling set $S\subseteq\rd$ 
for $PW_{\Omega}(H_{a})$ must have density 
$$
D^{-}(S)\geq\frac{\abs{\Sigma_{\Omega}^{b}}}{\left(2\pi\right)^{d}}
=\frac{\abs{B_{1}}}{\left(2\pi\right)^{d}}\det(b^{-1/2})\,\Omega^{d/2}\, , 
$$ and a set of interpolation $S$  in $\PW$$\left(H_{a}\right)$ 
must satisfy
$
D^{+}(S)\leq\frac{\abs{\Sigma_{\Omega}^{b}}}{\left(2\pi\right)^{d}}=\frac{\abs{B_{1}}}{\left(2\pi\right)^{d}}\det(b^{-1/2})\,\Omega^{d/2}$. 
This is Theorem \nameref{tm2} of the introduction. As was to be
expected,  this coincides with the critical density for the classical
Paley Wiener space $\PW(H_{b})$ (cf. \cite{landau67}).   

(ii) \emph{Symbols with radial limits.} Let us consider the class of symbols that possess radial limits at
$\infty$. We say that $a\in\cbinfd$ is spherically continuous, if it
possesses radial limits in the following sense.   There
exists a continuous matrix function $b \in C(S^{d-1}, \bC ^{d\times d}
  )$, such that
$$\lim _{r\to \infty} \sup _{\eta \in S^{d-1}} \|a(r\eta ) - b(\eta
)\|_X = 0 \, .
$$ 
A $3\epsilon$-argument shows that these symbols are slowly
oscillating.  Consequently Theorem \ref{mgen} holds for spherically continuous
symbols in $\cbinfd.$ Spherically continuous symbols are related to
another compactification, the spherical compactification with corona
$S^{d-1}$. In contrast to the Higson compactification, it is
metrizable, but it is much smaller. See~\cite{Cordes79} for its use in
partial differential equations. 
\end{example}

\subsection{Variable bandwith in dimension $d=1$}

Let $H_{a}$ be the differential operator $H_{a}f=-\frac{d}{dx}(a\frac{d}{dx}f)$
on $L^{2}(\bR)$. This is a Sturm-Liouville operator on $\bR$, and
the ellipticity assumption amounts to the conditions $\inf_{x\in\bR}a(x)>0$
and $a\in\cbinf$. In \cite{cpam} we argued that the spectral subspaces
of $H_{a}$ can be interpreted as spaces of locally variable bandwidth.
Intuitively, the quantity $a(x)^{-1/2}$ is a measure for the bandwidth
in a neighborhood of $x$. 

We apply Theorem~\ref{mgen} to $H_{a}$. The relevant measure is
$d\nu(x)=a^{-1/2}(x)dx$, and $\nu(I)=\int_{I}a(x)^{-1/2}\,dx$ for
$I\subseteq\bR$. 
Then we have the following necessary density condition for functions
of variable bandwidth (Corollary \nameref{cor1} of the introduction). 
\begin{cor}
\label{corvbw} Assume that $a\in C_{b}^{\infty}(\bR)$ and $\lim_{x\to\pm\infty}a'(x)=0$
. Let $\PW(H_{a})$ be the Paley-Wiener space associated to $H_{a}$.

(i) If $S$ is a sampling set for $\PW(H_{a})$, then 
\begin{equation}
D_{\nu}^{-}(S)=\liminf_{r\to\infty}\inf_{x\in\bR}\frac{\#(S\cap[x-r,x+r])}{\nu([x-r,x+r])}\geq\frac{\Omega^{1/2}}{\pi}\,.\label{eq:1dud}
\end{equation}
(ii) If $S$ is a set of interpolation for $\PW(H_{a})$, then $D_{\nu}^{+}(S)\leq\Omega^{1/2}/\pi$.
\end{cor}

Arguing as in Lemma~\ref{lem:changemeas}, equation~\eqref{eq:1dud}
says that for $\epsilon>0$ and $R$ large enough we have 
\begin{equation}
\#(S\cap[x-r,x+r])\geq\Big(\frac{\Omega^{1/2}}{\pi}-\epsilon\Big)\int_{x-r}^{x+r}a(y)^{-1/2}\,dy\,.\label{eq:c18}
\end{equation}
Thus the number of samples in an interval $[x-r,x+r]$ is determined
by $a(x)^{-1/2}$, which is in line with our interpretation of $a^{-1/2}$
as the local bandwidth.

Corollary~\ref{corvbw} is precisely the formulation of the necessary
density conditions in \cite{cpam}. However, the main result of \cite{cpam}
was proved under the restrictive assumption that $a$ is constant
outside an interval $[-R,R]$. The proof in \cite{cpam} dwelt heavily
on the scattering theory of one-dimensional Schrödinger operators.
The method of this paper yields a significantly more general result
with a completely different method of proof. Corollary~\ref{corvbw}
was our dream that motivated this work.

Finally we remark that the density conditions of Theorem~\ref{mgen}
suggest that the Paley-Wiener spaces $\PW(H_{a})$ associated to a
uniformly elliptic differential operator may be taken as an appropriate
generalization of variable bandwidth to higher dimensions.

\section{Outlook\label{sec:Outlook}}

We have proved necessary density conditions for sampling and interpolation
in spectral subspaces of uniformly elliptic partial differential operators
with slowly oscillating coefficients. These spectral subspaces may
be taken as a suitable generalization of the notion of variable bandwidth
to higher dimensions. The emphasis has been on a new method that combines
elements from limit operators, regularity theory and heat kernel estimates,
and the use of compactifications.

Clearly one can envision manifold extensions of our results and methods.
Theorem~\ref{mainabstr} is stated for a significantly larger class
of operators and symbols. For instance, it could be applied to higher
order partial differential operators or to Schrödinger operators and
to symbols with less smoothness or to almost periodic symbols. However,
the spectral theory of such operators is more involved and one needs
to find conditions that prevent their limit operators from having
a point spectrum at the ends of the spectral interval. As these questions
belong to spectral theory rather than sampling theory, we plan to
pursue them in a separate publication.

In a different direction one may consider the graph Laplacian on an
infinite graph or even a metric measure space endowed with a kernel
that satisfies Gaussian estimates~\cite{MR2970038}. While many steps
of our proofs remain in place, this set-up opens numerous new questions.

Finally several hidden connections beg to be explored. The identity
\eqref{eq:weyllike} resembles the famous Weyl formula for the asymptotic
density of eigenvalues in a spectral interval~\cite{hoermander68}.
This observation invites the comparison of the Beurling density with
the density of states in spectral theory. We plan to investigate some
of these issues in future work.

\appendix

\section{Averaged traces}

For completeness we provide the proof of Lemma~\ref{lem-sphint-1}.
Recall that 
$\tr^{-}\left(f\right)=$ \\
 $\liminf_{r\to\infty}\inf_{y\in\rd}\frac{1}{\abs{B_{r}\left(y\right)}}\int_{B_{r}\left(y\right)}f\left(x\right)\,dx$.

\emph{If $f,g\in C_{b}\left(\rd\right)$ and $\lim_{\abs x\to\infty}|f(x)-g(x)|=0$,
then} 
\[
\tr^{-}\left(f\right)=\tr^{-}\left(g\right)\quad\text{and}\quad\tr^{+}\left(f\right)=\tr^{+}\left(g\right)\,.
\]
\begin{proof}
Set $h=f-g$, then $\lim_{x\to\infty}h(x)=0$. We split the relevant
averages as 
\[
\frac{1}{\abs{B_{r}\left(y\right)}}\int_{B_{r}\left(y\right)}\abs{h\left(x\right)}dx=\frac{1}{\abs{B_{r}\left(y\right)}}\left[\int_{B_{r}\left(y\right)\cap B_{R}^{c}\left(0\right)}+\int_{B_{r}\left(y\right)\cap B_{R}\left(0\right)}\right]\abs{h\left(x\right)}dx=(I)+(II)\,.
\]
Given $\epsilon>0$, there exists an $R_{\varepsilon}>0$ such that
$\sup_{|x|\ge R_{\epsilon}}\abs{h\left(x\right)}<\varepsilon/2$.
So, 
\[
(I)<\frac{\abs{B_{r}\left(y\right)\cap B_{R_{\varepsilon}}^{c}\left(0\right)}}{\abs{B_{r}\left(y\right)}}\,\frac{\varepsilon}{2}<\frac{\varepsilon}{2}\,,
\]
independent of $y$. For the second term observe that 
\[
(II)<\norm[\infty]{h}\frac{\abs{B_{r}\left(y\right)\cap B_{R_{\varepsilon}}\left(0\right)}}{\abs{B_{r}\left(y\right)}}\leq\norm[\infty]{h}\frac{|B_{R_{\epsilon}}|}{|B_{r}(y)|}<\varepsilon/2
\]
for $r>\left(\frac{2}{\varepsilon}\norm[\infty]{h}\right)^{1/d}R_{\varepsilon}$.

Since 
\[
\frac{1}{\abs{B_{r}\left(y\right)}}\int_{B_{r}\left(y\right)}f=\frac{1}{\abs{B_{r}\left(y\right)}}\int_{B_{r}\left(y\right)}g+\frac{1}{\abs{B_{r}\left(y\right)}}\int_{B_{r}\left(y\right)}\left(f-g\right)\,,
\]
it follows that 
\[
\inf_{y\in\rd}\frac{1}{\abs{B_{r}\left(y\right)}}\int_{B_{r}\left(y\right)}f\leq\inf_{y\in\rd}\frac{1}{\abs{B_{r}\left(y\right)}}\int_{B_{r}\left(y\right)}g+\sup_{y\in\rd}\frac{1}{\abs{B_{r}\left(y\right)}}\int_{B_{r}\left(y\right)}\abs{f-g}\,,
\]
and therefore 
\begin{align*}
\tr^{-}\left(f\right) & \leq\liminf_{r\to\infty}\inf_{y\in\rd}\frac{1}{\abs{B_{r}\left(y\right)}}\int_{B_{r}\left(y\right)}g+\lim_{r\to\infty}\sup_{y\in\rd}\frac{1}{\abs{B_{r}\left(y\right)}}\int_{B_{r}\left(y\right)}\abs{f-g}=\tr^{-}\left(g\right)\,.
\end{align*}
Interchanging $f$ and $g$ yields equality. The equality $\mathrm{tr}^{+}(f)=\mathrm{tr}^{+}(g)$
is proved in the same way. 
\end{proof}

\section{The lower bound for the reproducing kernel}

We verify the lower bound for $\norm{k_{x}}$ in the proof of Proposition
\ref{prop:kernbound}. This fact is proved in \cite[Lemma 3.19 (a)]{MR2970038}.
For completeness we reproduce that proof with some necessary modifications
and adjustments. The idea is to relate the reproducing kernel to the
heat kernel of $e^{-tH_{a}}$ via functional calculus.



We write $k^{\Omega}$ for the reproducing kernel of $\PW$ $\left(H_{a}\right)$.
For a bounded, non-negative Borel function $F\geq0$ with support
in $[0,\Omega]$ we define $k_{x}^{F}=F\left(H_{a}\right)k_{x}^{\Omega}$
and the corresponding integral kernel $k^{F}\left(x,y\right)\coloneqq F\left(H_{a}\right)(x,y):=k_{x}^{F}\left(y\right)$.
The last expression is well defined, as $k_{x}^{F}\in\PW$$\left(H_{a}\right)$.
The kernel $k^{F}\left(x,y\right)$ is symmetric, because $F(H_{a})$
is self-adjoint: 
\[
k_{x}^{F}\left(y\right)=\inprod{F\left(H_{a}\right)k_{x}^{\Omega},k_{y}^{\Omega}}=\inprod{k_{x}^{\Omega},F\left(H_{a}\right)k_{y}^{\Omega}}=\overline{\inprod{F\left(H_{a}\right)k_{y}^{\Omega},k_{x}^{\Omega}}}=\overline{k_{y}^{F}\left(x\right)}\,.
\]
Consequently, $F(H_{a})$ is an integral operator. For $f\in\Lii$
\[
F\left(H_{a}\right)f\left(x\right)=\inprod{F\left(H_{a}\right)f,k_{x}^{\Omega}}=\inprod{f,F\left(H_{a}\right)k_{x}^{\Omega}}=\inprod{f,k_{x}^{F}}=\int_{\rd}k^{F}\left(y,x\right)f(y)dy\,.
\]
If $0\leq G\leq F$, then 
\begin{equation}
0\leq k^{G}\left(x,x\right)\leq k^{F}\left(x,x\right)\label{eq:c16}
\end{equation}
for all $x\in\rd$. For the proof observe that $F-G\geq0$ implies
that 
\[
k^{F}(x,x)=\inprod{F\left(H_{a}\right)k_{x}^{\Omega},k_{x}^{\Omega}}\geq\inprod{G\left(H_{a}\right)k_{x}^{\Omega},k_{x}^{\Omega}}=k^{G}(x,x)\,.
\]

The heat operator $e^{-tH_{a}}$ is bounded and has a kernel $p_{t}\left(x,y\right)$
that satisfies \emph{on diagonal estimates.} There are positive constants
$c,C$ such that for all $x\in\rd$, and $t>0$ 
\[
ct^{-d/2}\leq p_{t}\left(x,x\right)\leq Ct^{-d/2}\,.
\]
This is well known, see, e.g. \cite{Ouhabaz2006}. 
\begin{claim*}
$0<c<k^{\Omega}\left(x,x\right)<C$ for all $x\in\rd$. 
\end{claim*}

\begin{proof}
As $\ind_{\Omega}\left(u\right)\leq e\cdot e^{-u/\Omega}\chi_{[0,\infty)}\left(u\right)$
 and $\chi_{[0,\infty)}\left(H_{a}\right)=\mathrm{Id}$, 
we obtain 
\begin{equation}
k^{\Omega}(x,x)=\ind_{[0,\Omega]}(H_{a})(x,x)\leq ee^{-\Omega^{-1}H_{a}}(x,x)\le C\Omega^{d/2},\label{eq:compker}
\end{equation}
which gives an explicit upper bound for $\norm{k_{x}}^{2}$ in Proposition~\ref{prop:kernbound}.

For the proof of the lower bound, we use a dyadic decomposition. 
\begin{align}
\chi_{[0,\infty)}\left(u\right)e^{-tu} & =\ind_{[0,\Omega]}(u)e^{-tu}+\sum_{k\geq0}\ind_{(2^{k}\Omega,2^{k+1}\Omega]}(u)e^{-tu}\label{eq:tailest}\\
 & \leq\ind_{[0,\Omega]}(u)+\sum_{k\geq0}\ind_{[0,2^{k+1}\Omega]}(u)e^{-t2^{k}\Omega}\qquad t>0\,.\nonumber 
\end{align}
One can verify that this inequality remains true as an operator inequality
\[
\chi_{[0,\infty)}\left(H_{a}\right)e^{-tH_{a}}\leq\ind_{[0,\Omega]}(H_{a})+\sum_{k\geq0}\ind_{[0,2^{k+1}\Omega]}(H_{a})e^{-t2^{k}\Omega}\,,
\]
with strong convergence of the sum, and every term is an integral
operator. 
By~\eqref{eq:c16} and \eqref{eq:compker} the operator inequality
can be transferred to a corresponding inequality of the diagonals
of the integral kernel as follows: 
\begin{align*}
ct^{-d/2} & \le p_{t}(x,x)\\
 & \le\ind_{[0,\Omega]}(H_{a})(x,x)+\sum_{k\geq0}\ind_{[0,2^{k+1}\Omega]}(H_{a})(x,x)e^{-t2^{k}\Omega}\\
 & \le\ind_{[0,\Omega]}(H_{a})(x,x)+C\Omega^{d/2}\sum_{k\geq0}2^{(k+1)d/2}\,e^{-t2^{k}\Omega}
\end{align*}

We choose $t=2^{r}/\Omega$ for $r\in\n$ to be specified later. Then

\begin{align*}
c\Omega^{d/2}2^{-rd/2} & \le\ind_{[0,\Omega]}(H_{a})(x,x)+C\Omega^{d/2}\sum_{k\geq0}e^{-2^{k}2^{r}}2^{(k+1)d/2}\\
 & =\ind_{[0,\Omega]}(H_{a})(x,x)+C2^{d/2}\Omega^{d/2}2^{-rd/2}\sum_{k\geq0}e^{-2^{k+r}}2^{(k+r)d/2}\\
 & \leq\ind_{[0,\Omega]}(H_{a})(x,x)+C2^{d/2}\Omega^{d/2}2^{-rd/2}\sum_{k\geq r}e^{-2^{k}}2^{kd/2}.
\end{align*}
Hence, 
\[
\Omega^{d/2}2^{-rd/2}\Big(c-C'2^{d/2}\sum_{k\geq r}e^{-2^{k}}2^{kd/2}\Big)\le\ind_{[0,\Omega]}(H_{a})(x,x)=k^{\Omega}(x,x)
\]
For $r\in\n$ sufficiently large, this implies the lower bound for
$k^{\Omega}(x,x)$. 
\end{proof}
\bibliographystyle{amsalpha}
\bibliography{almost_diagonal,general,new}

\end{document}